\documentclass[reqno]{amsart}
\usepackage[utf8]{inputenc}
\usepackage{amsmath}
\usepackage{amssymb}
\usepackage{amsthm}
\usepackage[all]{xy}
\xyoption{tips}
\SelectTips{cm}{10}
\usepackage[dvipsnames,table,xcdraw]{xcolor}
\usepackage{fullpage}
\usepackage{stmaryrd}
\usepackage{enumitem}
\usepackage{mathtools}
\usepackage{microtype}
\usepackage{todonotes,booktabs}
\usepackage{float}

\DeclareMathAlphabet{\mathpzc}{OT1}{pzc}{m}{it}

\usepackage[english]{babel}
\usepackage{enumerate}
\usepackage{hyperref}
\usepackage{pgf,tikz}
\usepackage{tikz-cd}
\hypersetup{colorlinks=true,citecolor=red, linkcolor=blue}

\usepackage[titletoc]{appendix}
\usepackage{cleveref}

\usepackage{graphicx}
\usepackage{mathrsfs}
\usepackage{setspace}
\usepackage{nccmath}
\usepackage{multirow}

\usetikzlibrary{arrows}
\usetikzlibrary{cd}
\usetikzlibrary{babel}
\usetikzlibrary{trees}

\newcommand{\Z}{\mathbb{Z}}

\newcommand{\OO}{{\mathcal O}}

\newcommand{\CC}{{\mathcal C}}
\newcommand{\DD}{{\mathcal D}}

\newcommand{\calA}{\mathcal{A}}

\newcommand{\uC}{\underline{\mathscr{C}}}
\newcommand{\uD}{\underline{\mathscr{D}}}
\newcommand{\VV}{\mathscr{V}}

\newcommand{\umu}{\underline{\mu}}
\newcommand{\umui}{\underline{\mu}^{-1}}

\newcommand{\s}{\mathfrak{s}}

\newcommand{\col}{\mathrm{Col}}

\newcommand{\tc}{\mathrm{tC}_R}
\newcommand{\fmod}{\mathrm{fMod}_R}
\newcommand{\fc}{\mathrm{fC}_R}
\newcommand{\vbOp}{\mathrm{vbOp}}
\newcommand{\fCOp}{\mathrm{fCOp}}

\newcommand{\bgmod}{\mathrm{bgMod}_R}
\newcommand{\vbc}{\mathrm{vbC}_R}

\newcommand{\ubgMod}{\underline{\mathpzc{bgMod}_R}}
\newcommand{\ufMod}{\underline{\mathpzc{fMod}_R}}

\newcommand{\uEnd}{\underline{\mathpzc{End}}}
\newcommand{\utC}{\underline{t\mathcal{C}_R}}
\newcommand{\ufC}{\underline{\mathpzc{fC}_R}}
\newcommand{\usfC}{\underline{\mathpzc{sfC}_R}}
\newcommand{\Tot}{\mathrm{Tot}}
\newcommand{\vdeg}{\mathrm{vdeg}}

\DeclareMathOperator{\Ima}{Im}
\DeclareMathOperator{\Hom}{Hom}
\DeclareMathOperator{\End}{End}

\newtheorem{manualtheoreminner}{Theorem}
\newenvironment{manualtheorem}[1]{%
  \renewcommand\themanualtheoreminner{#1}%
  \manualtheoreminner
}{\endmanualtheoreminner}

\newtheorem{defin}{Definition}[section]

\newtheorem{lem}[defin]{Lemma}
\newtheorem{propo}[defin]{Proposition}
\newtheorem{thm}[defin]{Theorem}
\newtheorem{corollary}[defin]{Corollary}

\theoremstyle{remark}
\newtheorem{remark}[defin]{Remark}

\DeclareMathSymbol{\mathinvertedexclamationmark}{\mathclose}{operators}{'074}
\DeclareMathSymbol{\mathexclamationmark}{\mathclose}{operators}{'041}

\makeatletter
\newcommand{\raisedmathinvertedexclamationmark}{%
  \mathclose{\mathpalette\raised@mathinvertedexclamationmark\relax}%
}
\newcommand{\raised@mathinvertedexclamationmark}[2]{%
  \raisebox{\depth}{$\m@th#1\mathinvertedexclamationmark$}%
}
\begingroup\lccode`~=`! \lowercase{\endgroup
  \def~}{\@ifnextchar`{\raisedmathinvertedexclamationmark\@gobble}{\mathexclamationmark}}
\mathcode`!="8000
\makeatother

\begin{document}

\title{The derived Deligne conjecture}
\author{Javier Aguilar Mart\'in}
\author{Constanze Roitzheim}
\address{University of Kent, School of Mathematics, Statistics and Actuarial Science, Canterbury, CT2 7NF, UK}
\email{c.roitzheim@kent.ac.uk}

\begin{abstract}
Derived $A_\infty$-algebras have a wealth of theoretical advantages over regular $A_\infty$-algebras. However, due to their bigraded nature, in practice they are often unwieldy to work with. We develop a framework involving brace algebras on operads which allows us to study derived $A_\infty$ algebras in a new conceptual context. One particular advantage is that this construction allows us to generalize the Lie algebra structure on the Hochschild complex of an $A_\infty$-algebra, obtaining new and rigorous versions of the Deligne conjecture.
\end{abstract}

\maketitle


\section{Introduction}

There are a number of mathematical fields in which $A_\infty$-structures arise naturally, ranging from topology to mathematical physics. To study these structures, different interpretations of $A_\infty$-algebras can be of use. From the original definition in \cite{STASHEFF} to alternative definitions in terms of tensor coalgebras \cite{keller}, \cite{penkava}, many approaches use the machinery of operads \cite{LRW}, \cite{lodayvallette} or certain Lie brackets \cite{RW} to obtain these objects. 

Another technique to describe $A_\infty$-structures comes from \emph{brace algebras} \cite{GV},\cite{lada}, which often involves big calculations that are difficult to handle.
In this article, we use an operadic approach to deal with the relationship between brace algebras and $A_\infty$-algebras in a conceptual manner, laying down rigorous foundations which allow us to make better use of the potential that brace structures have to offer, as well as to generalise existing approaches such as the Lie bracket methods in \cite{RW}. 

One advantage to our approach is is that it only requires relatively elementary tools and that it can be used to talk about $A_\infty$-structrures on any operad, which provides a useful way of thinking about $A_\infty$-structures.  A first application is the generalization of the \emph{Deligne conjecture}. The {classical Deligne conjecture}  has its roots in the theory of topological operads \cite{delignehistory}. It implies that the Hochschild complex of an associative algebra has the structure of a homotopy $G$-algebra \cite{GV}, which we prove as \Cref{theorem}.

\begin{manualtheorem}{A}
The brace algebra structure on an operad $\OO$ gives rise to a morphism of $A_\infty$-algebras $\Phi:S\s\OO\to S\s\End_{S\s\OO}$.
\end{manualtheorem} 
This result was hinted at by Gerstenhaber and Voronov in \cite{GV}, but here we introduce a suitable context and prove it as \Cref{theorem}.

Since $A_\infty$- algebras generalize associative algebras, it is natural to ask what sort of algebraic structure is carried by their Hochschild complex. Thanks to the tools we develop, we are able to answer this question, obtaining the following $A_\infty$-version of the Deligne conjecture in \Cref{ainftydeligne}. 

\begin{manualtheorem}{B}
The Hochschild complex $S\s\End_{S\s\OO}$ of an operad with an $A_\infty$-multiplication has a structure of a $J$-algebra.
\end{manualtheorem} 
 In the above theorem, $J$-algebras provide the appropriate generalization of homotopy $G$-algebras from the classical case \cite{GV}.

In the 2000s, derived $A_\infty$-algebras were introduced as a bigraded generalization of $A_\infty$-algebras in order to bypass the restrictive projectivity assumptions that are often required when working with classical $A_\infty$-algebras. The key difference is that derived $A_\infty$-algebras carry an additional degree which allows for internal projective resolutions. Of course, carrying another degree comes with multiple practical challenges, creating the need for particularly careful theoretical foundations in order to generalize the Deligne conjecture theorems to the derived context. We start by showing that indeed, under mild boundedness assumptions, any operad $\OO$ with a derived $A_\infty$-multiplication carries linear maps $M_{ij}:(S\s\OO)^{\otimes j}\to S\s\OO$, satisfying the derived $A_\infty$-algebra axioms.

 The next result is \Cref{bigradedtheorem}, which generalizes \Cref{theorem} to the derived setting. 
 \begin{manualtheorem}{C}
The brace algebra structure on a suitable operad $\OO$ gives rise to a morphism of derived $A_\infty$-algebras $\Phi:S\s\OO\to S\s\End_{S\s\OO}$.
\end{manualtheorem}
As a consequence of this theorem we obtain a new version of the Deligne conjecture, \Cref{dainftydeligne}. 
For this we also develop a derived version of $J$-algebras.
\begin{manualtheorem}{D}
The Hochschild complex $S\s\End_{S\s\OO}$ of an operad with a derived $A_\infty$-multiplication has a structure of derived $J$-algebra.
\end{manualtheorem} 

Our results therefore develop new strong links between several key structures in algebraic operads and open up the potential to apply what was previously only possible in a single-graded, projective setting to a wider topological context.

\bigskip

The paper is organized as follows. \Cref{sec:bigraded}, \Cref{sec:total} and \Cref{sec:enrichment} establish necessary notation and conventions with regards to the bigraded and enriched categories we use. It is of crucial importance that we take great care laying down these foundations due to the potentially messy nature of working in a bigraded context.
With all this in place, we finally we move on to a recap of derived $A_\infty$-algebras in \Cref{sec:deriveddef}.
The next important ingredients are a suitable bigraded version of operadic suspension, which we introduce in  \Cref{sec:operadic}, and brace structures, introduced in \Cref{sec:braces}.
We then step back to the single graded case of the classical Deligne conjecture in \Cref{sec:classicaldeligne}. As the derived version depends on rigorous groundwork with regards to the single-graded case, it is important for our purposes to explicitly present this case.
In \Cref{sec:derivedstructure} we then apply all our categorical tools to obtain the derived statements finally proving the derived Deligne conjecture in \Cref{sec:ddeligne}.

\section*{Acknowledgements} This article originated as the PhD thesis of the first author under the supervision of the second author. Both authors thank the University of Kent, particularly for providing the first author with a Graduate Teaching Assistantship. We would furthermore like to express special thanks Sarah Whitehouse for her support, especially during the second author's maternity leave.

\section{Filtered Modules and complexes}\label{sec:bigraded}

In this section we will collect the necessary notions for the base categories that we will be working in. Derived $A_\infty$-categories exist in a bigraded setting, so it is only natural that we will also present conventions on grading and signs.

Let $\CC$ be a category and let $A$, $B$ be 
objects in $\CC$. We denote by $\Hom_\CC(A,B)$ the set of morphisms from $A$ to $B$ in $\CC$. If $(\CC,\otimes, 1)$ is a
closed symmetric monoidal category, then we denote its internal hom-object by $[A,B] \in\CC$.

Let us now collect some definitions with regards to filtered modules and filtered complexes. Filtrations will allow to add an extra degree to single-graded objects which will be a key technique for relating them to bigraded objects. 

\begin{defin}
A \emph{filtered $R$-module} $(A, F)$ is given by a family of $R$-modules $\{F_pA\}_{p\in\Z}$ indexed by
the integers such that $F_{p}A \subseteq F_{p-1}A$ for all $p \in\Z$ and $A = \bigcup_p F_pA$. A \emph{morphism of filtered modules} is a morphism $f : A \to B$ of $R$-modules which is compatible with filtrations: $f(F_pA) \subset F_pB$ for all $p \in\Z$.
\end{defin}
We denote by $\mathrm{C}_R$ the category of cochain complexes of $R$-modules.
\begin{defin}\label{filteredcomplex}
A \emph{filtered complex} $(K, d, F)$ is a complex $(K, d) \in\mathrm{C}_R$ together with a filtration $F$ of each $R$-module $K^n$ such that $d(F_pK^n) \subset F_pK^{n+1}$ for all $p, n \in\Z$. Its morphisms are given by
morphisms of complexes $f : K \longrightarrow L$ compatible with filtrations.
\end{defin}

We denote by $\fmod$ and $\fc$ the categories of filtered modules and filtered complexes of $R$-modules, respectively.

\begin{defin}\label{filteredtensor}
The \emph{tensor product of two filtered $R$-modules} $(A, F)$ and $(B, F)$ is the filtered $R$-module
with
 \[F_p(A \otimes B) :=\sum_{i+j=p}\Ima(F_iA \otimes F_jB \longrightarrow A \otimes B).\]
This makes the category of filtered $R$-modules into a symmetric monoidal category, where the unit is given by $R$ with the trivial filtration $0 = F_{1}R \subset F_0R = R$.
\end{defin}

\begin{defin}\label{filterend}
Let $K$ and $L$ be filtered complexes. We define $\underline{\Hom}(K,L)$ to be the filtered complex whose underlying cochain complex is $\Hom_{\mathrm{C}_R}(K,L)$ and the filtration $F$ given by 
\[F_p\underline{\Hom}(K,L)=\{f:K\to L\mid f(F_qK)\subset F_{q+p}L\text{ for all }q \in\Z\}.\]
In particular, $\Hom_{\fmod}(K,L)=F_0\underline{\Hom}(K,L)$.
\end{defin}

\begin{defin}
We consider $(\Z,\Z)$-bigraded
$R$-modules $A = \{A^j_i\}$, where elements of $A^j_i$ are said to have bidegree $(i, j)$. We sometimes refer to $i$
as the \emph{horizontal} degree and $j$ the \emph{vertical degree}. The \emph{total degree} of an element $x \in A^j_i$ is $i+j$ and is denoted by $|x|$.
\end{defin}

We introduce the following scalar product notation for bidegrees: for $x$, $y$ of bidegree $(x_1, x_2)$, $(y_1, y_2)$
respectively, we let $\langle x, y\rangle = x_1y_1 + x_2y_2$.

\begin{defin}
A \emph{morphism of bidegree $(p, q)$} maps $A^j_i$ to $A^{j+q}_{i+p}$. The tensor product of two bigraded $R$-modules $A$
and $B$ is the bigraded $R$-module $A \otimes B$ given by
\[(A \otimes B)^j_i \coloneqq \bigoplus_{p,q}A^q_p \otimes B^{j-q}_{i-p} .\]
\end{defin}
We denote by $\bgmod$ the category whose objects are bigraded $R$-modules and whose morphisms
are morphisms of bigraded $R$-modules of bidegree $(0, 0)$. It is symmetric monoidal with the above
tensor product.
The symmetry isomorphism is given by
\[\tau_{A\otimes B} : A \otimes B \longrightarrow B \otimes A,\ x \otimes y \mapsto (-1)^{\langle x,y\rangle}y \otimes x.\]
We follow the \emph{Koszul sign rule:} if $f : A \longrightarrow B$ and $g : C \longrightarrow D$ are bigraded morphisms, then the
morphism $f \otimes g : A \otimes C \longrightarrow B \otimes D$ is defined by
\[(f \otimes g)(x \otimes z) \coloneqq (-1)^{\langle g,x\rangle}f(x) \otimes g(z).\]

The following categories will occur naturally throughout our work.
\begin{defin}
A \emph{vertical bicomplex} is a bigraded $R$-module $A$ equipped with a vertical differential \linebreak $d^A : A \longrightarrow A$ of bidegree $(0, 1)$. A \emph{morphism of vertical bicomplexes} is a morphism of bigraded modules
of bidegree $(0, 0)$ commuting with the vertical differential. The category of vertical bicomplexes is denoted by $\vbc$.
\end{defin}

 The tensor product of two vertical bicomplexes $A$ and $B$ is given by endowing the tensor product of underlying bigraded modules with
vertical differential \[d^{A\otimes B} := d^A \otimes 1 + 1 \otimes d^B : (A \otimes B)^v_u \longrightarrow (A \otimes B)^{v+1}_u .\] This makes $\vbc$ into a
symmetric monoidal category.

The symmetric monoidal categories $(\mathrm{C}_R,\otimes,R)$, $(\bgmod,\otimes,R)$ and $(\vbc,\otimes,R)$ are related by embeddings $\mathrm{C}_R\to\vbc$ and $\bgmod \to\vbc$ which are monoidal and full.

\begin{defin}\label{delta1}
Let $A,B$ be bigraded modules. We define $[A,B]^\ast_\ast$
to be the bigraded module of morphisms of bigraded modules $A \longrightarrow B$. Furthermore, if $A,B$ are vertical bicomplexes, and $f \in
[A,B]^v_u$, we define
\[\delta (f) := d_Bf - (-1)^vfd_A.\]
\end{defin}

Direct computation shows the following. 

\begin{lem}
If $A$, $B$ are vertical bicomplexes, then $([A,B]^\ast_\ast
, \delta )$ is a vertical bicomplex. \qed
\end{lem}

\begin{defin}\label{twistedcomplex} 
The category of twisted complexes $\tc$ is defined as follows. 
The objects are \emph{twisted complexes}, where a twisted complex $(A, d_m)$ is a bigraded $R$-module $A = \{A^j_i \}$ together with a family
of morphisms $\{d_m : A \rightarrow A\}_{m\geq0}$ of bidegree $(m,1-m )$ such that for all $m \geq 0$,
\[\sum_{i+j=m}(-1)^id_id_j = 0.\]

A \emph{morphism of twisted complexes} $f : (A, d^A_m) \longrightarrow (B, d^B_m)$ is given by a family of morphisms of $R$-modules $\{f_m : A \longrightarrow B\}_{m\geq0}$ of bidegree $(m,-m)$ such that for all $m \geq 0$,
\[\sum_{i+j=m}d^B_if_j =\sum_{i+j=m}(-1)^if_id^A_j.\]
The composition of morphisms is given by $(g \circ f)_m :=\sum_{i+j=m} g_if_j$.

A morphism $f = \{f_m\}_{m\geq0}$ is
 \emph{strict} if $f_i = 0$ for all $i > 0$. The \emph{identity} morphism $1_A : A \longrightarrow A$ is the strict morphism
given by $(1_A)_0(x) = x.$ A morphism $f = \{f_i\}$ is an \emph{isomorphism} if and only if $f_0$ is an isomorphism of
bigraded $R$-modules. 
\end{defin}
Note that if $f$ is an isomorphism, then an inverse of $f$ is obtained from an inverse of $f_0$ by solving a triangular system of linear equations.

\bigskip
We finish this section by recalling the symmetric monoidal structure on $\tc$.
\begin{lem}[{{\cite[Lemma 3.3]{whitehouse}}}]\label{tensortwisted}
The category $(\tc,\otimes,R)$ is symmetric monoidal, where the monoidal structure is given
by the bifunctor
\[\otimes : \tc \times \tc \longrightarrow \tc.\]
On objects it is given by $((A, d^A_m), (B, d^B_m)) \longrightarrow (A \otimes B, d^A_m \otimes 1 + 1 \otimes d^B_m)$ and on morphisms it is
given by $(f, g) \longrightarrow f \otimes g$, where $(f \otimes g)_m :=\sum_{i+j=m} f_i \otimes g_j$. In particular, by the Koszul sign rule we
have that \[(f_i \otimes g_j)(x\otimes z) = (-1)^{\langle g_j ,x\rangle}f_i(x)\otimes g_j(z).\] The symmetry isomorphism is given by the strict
morphism of twisted complexes
\[
\tau_{A\otimes B} \colon A \otimes B \longrightarrow B \otimes A,\ x \otimes y\mapsto (-1)^{\langle x,y\rangle}y \otimes x.
\]
\end{lem}

The internal hom on bigraded modules can be extended to twisted complexes via the following lemma.

\begin{lem}[{{\cite[Lemma 3.4]{whitehouse}}}]\label{di} Let $A,B$ be twisted complexes. For $f \in[A,B]^v_u$, setting
\[(d_if) := (-1)^{i(u+v)}d^B_if - (-1)^vfd^A_i\]
for $i \geq 0$ endows $[A,B]^\ast_\ast$ with the structure of a twisted complex.
\end{lem}

\section{Totalization}\label{sec:total}

In order to relate derived $A_\infty$-algebras to the classical $A_\infty$-settings, we will make use of totalization techniques. 
We will introduce the specific totalization functor we are working with and the required related tools. This functor and its enriched versions are key to establish a correspondence between $A_\infty$-algebras and derived $A_\infty$-algebras.

\begin{defin}
The \emph{totalization} $\Tot(A)$ of a bigraded $R$-module $A = \{A^j_i \}$ is the graded $R$-module given by
\[\Tot(A)^n \coloneqq
\bigoplus_{i<0}A^{n-i}_i \oplus\prod_{i\geq 0}A^{n-i}_i .\]
The \emph{column filtration} of $\Tot(A)$ is the filtration given by \[F_p\Tot(A)^n \coloneqq\prod_{i\geq p} A^{n-i}_i .\]
\end{defin}

Given a twisted complex $(A, d_m)$, define a map $d : \Tot(A) \longrightarrow \Tot(A)$ of degree $1$ by letting
\[d(x)_j \coloneqq \sum_{m\geq0}(-1)^{mn}d_m(x_{j-m})\]
for $x = (x_i)_{i\in\Z} \in\Tot(A)^n$. Here $x_i \in A^{n-i}_i$ denotes the $i$-th component of $x$, and $d(x)_j$ denotes the $j$-th component of $d(x)$. Note
that, for a given $j \in\Z$ there is a sufficiently large $m \geq 0$ such that $x_{j-m'} = 0$ for all $m' \geq m$. Hence
$d(x)_j$ is given by a finite sum. Also, for negative $j$ sufficiently large, one has $x_{j-m} = 0$ for all $m \geq 0$, which
implies $d(x)_j = 0$.

Given a morphism $f : (A, d_m) \longrightarrow (B, d_m)$ of twisted complexes, let the \emph{totalization of $f$} be the map $\Tot(f) : \Tot(A) \longrightarrow \Tot(B)$ of degree 0 defined by
\[(\Tot(f)(x))_j \coloneqq \sum_{m\geq0}(-1)^{mn}f_m(x_{j-m})\]
 for $x = (x_i)_{i\in\Z} \in\Tot(A)^n$.

\begin{thm}[{{\cite[Theorem 3.8]{whitehouse}}}]
Let $F$ be the column filtration of $\Tot(A)$. The assignments $(A, d_m) \mapsto (\Tot(A), d, F)$
and $f \mapsto \Tot(f)$ define a functor $\Tot : \tc \to \fc$ which is an isomorphism of categories when restricted to its image.
\end{thm}

For a filtered complex of the form $(\Tot(A),d,F)$ where $A = \{A^j_i \}$ is a bigraded $R$-module, we can recover the twisted complex structure on  $A$ as follows. For all $m \geq 0$, let
$d_m : A \longrightarrow A$ be the morphism of bidegree $(m,1-m)$ defined by 
\[d_m(x) = (-1)^{nm}d(x)_{i+m},\] 
where $x \in A^{n-i}_i$ and $d(x)_k$ denotes the $k$-th component of $d(x)$. Note that $d(x)_k$ lies in $A^{n+1-k}_k$.

We will consider the following bounded categories since the totalization functor has better properties when restricted to them. 

\begin{defin}
We let $\tc^b$, $\vbc^b$ and $\bgmod^b$ be the full subcategories of \emph{horizontally bounded on the right} graded twisted
complexes, vertical bicomplexes and bigraded modules respectively. This means that if $A=\{A^j_i\}$ is an object of any of this categories, then there exists $i$ such that $A^j_{i'}=0$ for $i'>i$.

We let $\fmod^b$ and $\fc^b$ be the full subcategories of bounded filtered modules, respectively complexes, i.e.
the full subcategories of objects $(K, F)$ such that there exists some $p$ with the property that $F_{p'}K^n = 0$ for all $p'>p$. We refer to all of these as the \emph{bounded subcategories} of $\tc$, $\vbc$, $\bgmod$, $\fmod$ and $\fc$ respectively.
\end{defin}

\begin{propo}[{{\cite[Proposition 3.11]{whitehouse}}}]\label{monoidal}
The functors $\Tot : \bgmod \longrightarrow \fmod$ and $\Tot : \tc \longrightarrow \fc$ are lax symmetric
monoidal with structure maps
\[\epsilon : R \longrightarrow \Tot(R)\text{ and }\mu=\mu_{A,B} : \Tot(A) \otimes \Tot(B) \longrightarrow \Tot(A \otimes B)\]
given by $\epsilon = 1_R$. For $x = (x_i)_i \in\Tot(A)^{n_1}$ and  $y=(y_j)_j \in\Tot(B)^{n_2}$,
\begin{equation}\label{mu1}
\mu(x \otimes y)_k \coloneqq
\sum_{k_1+k_2=k}(-1)^{k_1n_2}x_{k_1} \otimes y_{k_2} .
\end{equation}

When restricted to the bounded case, $\Tot : \bgmod^b
 \longrightarrow \fmod^b$ and $\Tot : \tc^b \longrightarrow \fc^b$ are
strong symmetric monoidal functors.
\end{propo}

\begin{remark}\label{heuristic}
There is a certain heuristic to obtain the sign appearing in the definition of $\mu$ in \Cref{monoidal}. In the bounded case, we can write \[\Tot(A)=\bigoplus_i A_i^{n-i}.\]
As direct sums commute with tensor products, we have
\[\Tot(A)\otimes\Tot(B)=(\bigoplus A_i^{n-i})\otimes \Tot(B)\cong \bigoplus_i  (A_i^{n-i}\otimes \Tot(B)).\]

In the isomorphism we can interpret that each $A_i^{n-i}$ passes by $\Tot(B)$. Since $\Tot(B)$ used total grading, we can think of this degree as being the horizontal degree, while having 0 vertical degree. Thus, using the Koszul sign rule we would get precisely the sign from \Cref{monoidal}. This explanation is just an intuition, and opens the door for other possible sign choices: what if we decide to distribute $\Tot(A)$ over $\bigoplus_i B_i^{n-i}$ instead, or if we consider the total degree as the vertical degree? These alternatives lead to other valid definitions of $\mu$, and we will explore the consequences of some of them in \Cref{othermu}.
\end{remark}

\begin{lem}\label{mui}
In the conditions of \Cref{monoidal} for the bounded case, the inverse
\[\mu^{-1}:\Tot(A_{(1)}\otimes\cdots\otimes A_{(m)})\to \Tot(A_{(1)})\otimes\cdots\otimes \Tot(A_{(m)})\]
is given on pure tensors (for notational convenience) as
\begin{equation}\label{mu}
\mu^{-1}(x_{(1)}\otimes\cdots\otimes x_{(m)})=(-1)^{\sum_{j=2}^m n_j\sum_{i=1}^{j-1}k_i}x_{(1)}\otimes\cdots\otimes x_{(m)},
\end{equation}
where $x_{(l)}\in (A_{(m)})_{k_l}^{n_l-k_l}$.
\end{lem}
\begin{proof}
For the case $m=2$,
\[\mu^{-1}:\Tot(A\otimes B)\to \Tot(A)\otimes \Tot(B)\]
is computed explicitly as follows.
Let  $c\in\Tot(A\otimes B)^n$. By definition, we have
\[\Tot(A\otimes B)^n=\bigoplus_k (A\otimes B)^{n-k}_k=\bigoplus_k\underset{n_1+n_2=n}{\bigoplus_{k_1+k_2=k}}A_{k_1}^{n_1-k_1}\otimes B_{k_2}^{n_2-k_2}.\]
And thus, $c=(c_k)_k$ may be written as a finite sum $c=\sum_k c_k$, where 
\[c_k=\underset{n_1+n_2=n}{\sum_{k_1+k_2=k}}x_{k_1}^{n_1-k_1}\otimes y_{k_2}^{n_2-k_2}.\]
Here, we introduced superscripts to indicate the vertical degree, which, unlike in the definition of $\mu$ (\Cref{mu1}), is not solely determined by the horizontal degree since the total degree also varies. However we are going to omit them in what follows for simplicity of notation. Distributivity allows us to rewrite $c$ as
\[c=\sum_k \underset{n_1+n_2=n}{\bigoplus_{k_1+k_2=k}}x_{k_1}\otimes y_{k_2}=\sum_{n_1+n_2=n}\sum_{k_1}\sum_{k_2}(x_{k_1}\otimes y_{k_2})=\sum_{n_1+n_2=n}\left(\sum_{k_1}x_{k_1}\right)\otimes\left(\sum_{k_2}y_{k_2}\right).\]
Therefore, $\mu^{-1}$ can be defined as
\[\mu^{-1}(c)=\sum_{n_1+n_2=n}\left(\sum_{k_1}(-1)^{k_1n_2}x_{k_1}\right)\otimes\left(\sum_{k_2}y_{k_2}\right).\]

The general case follows inductively.
\end{proof}

\section{Enriched categories and enriched totalization}\label{sec:enrichment}

We collect some notions related to enriched categories from  \cite[\S 4.2]{whitehouse} that we will need as a categorical setting for our results on derived $A_\infty$-algebras. The purpose of this section is to introduce notation and the categories that we are working on. As such, it is quite dry in nature but it is both necessary and hopefully convenient for the reader. We assume the reader to be familiar with the basics of monoidal categories and enrichments, see e.g. \cite{riehl} for an excellent source. 

\begin{defin}
Let $(\VV ,\otimes, 1)$ be a symmetric monoidal category and let $(\CC,\otimes, 1)$ be a monoidal category. We say that $\CC$ is a \emph{monoidal category over $\VV$} if we have an external tensor product $\ast :\VV \times \CC \longrightarrow \CC$ with natural unit and associativity isomorphisms.

\end{defin}

\begin{remark}\label{underline}
We will also assume that there is a bifunctor $\uC(-,-) : \CC^{op} \times  \CC \longrightarrow \VV$ such that we have natural
bijections
\[\Hom_\CC(C \ast X, Y ) \cong \Hom_\VV (C,\uC(X, Y )).\]
Thus, we get a $\VV$-enriched category $\uC$ with the same objects as $\CC$ and with hom-objects given by $\uC (-,-)$. The unit
morphism $u_A : 1 \longrightarrow \uC (A,A)$ corresponds to the identity map in $\CC$ under the adjunction, and the
composition morphism is given by the adjoint of the composite
\[(\uC (B,C) \otimes \uC (A,B)) \ast A
\cong \uC (B,C) \ast (\uC (A,B) \ast A)
\xrightarrow{id\ast ev_{AB}}
\uC (B,C) \ast B
\xrightarrow{ev_{BC}} C,\]
where $ev_{AB}$ is the adjoint of the identity $\uC (A,B) \longrightarrow \uC (A,B)$. Furthermore, $\uC$ is a monoidal $\VV$-enriched category, namely we have an
enriched functor
\[\underline{\otimes} : \uC \times \uC \longrightarrow \uC\]
where $\uC \times \uC$ is the enriched category with hom-objects
$\uC \times \uC ((X, Y ), (W,Z)) \coloneqq \uC (X,W) \otimes \uC (Y,Z).$
In particular we get maps in $\VV$
\[\uC (X,W) \otimes \uC (Y,Z) \longrightarrow \uC (X \otimes Y,W \otimes Z),\]
given by the adjoint of the composite
\[(\uC (X,W) \otimes \uC (Y,Z)) \ast (X \otimes Y )\cong (\uC (X,W) \ast X) \otimes (\uC (Y,Z) \ast Y )
\xrightarrow{ev_{XW}\otimes ev_{Y Z}} W \otimes Z.\]
\end{remark}

\begin{defin}
Let $\CC$ and $\DD$ be monoidal categories over $\VV$. A \emph{lax functor over $\VV$} consists of a functor $F : \CC \longrightarrow \DD$ together with a natural transformation \[\nu_F : - \ast_\DD F(-) \Rightarrow F(- \ast_\CC -)\]
which is associative and unital with respect to the monoidal structures over $\VV$ of $\CC$ and $\DD$, see \cite[Proposition 10.1.5]{riehl} for  the explicit coherence axioms. If $\nu_F$ is a natural isomorphism,
we say $F$ is a \emph{functor over $\VV$}.
\end{defin}

Natural transformations over $\VV$ and (lax) monoidal functors over $\VV$ are defined analogously.

\begin{propo}\label{enrichedtrans}
Let $F,G : \CC \longrightarrow \DD$ be lax functors over $\VV$. Then $F$ and $G$ extend to $\VV$-enriched
functors
\[\underline{F},\underline{G} : \uC \longrightarrow \uD\]
where $\uC$ and $\uD$ denote the $\VV$-enriched categories corresponding to $\CC$ and $\DD$ as described in \Cref{underline}. Moreover, any natural transformation $\mu : F \Rightarrow G$ over $\VV$ also extends to a $\VV$-enriched natural
transformation
\[\underline{\mu} : \underline{F} \Rightarrow \underline{G}.\]
In particular, if $F$ is lax monoidal over $\VV$, then $\underline{F}$ is lax monoidal in the enriched sense, where the monoidal structure on $\uC \times \uC$ is described in \Cref{underline}.
\end{propo}

\begin{lem}
Let $F,G:\CC\to\DD$ lax functors over $\VV$ and let $\mu : F\Rightarrow G$ a natural transformation over $\VV$. For every $X\in\CC$ and $Y\in\DD$ there is a map \[\uD(GX,Y)\to\uD(FX,Y)\] that is an isomorphism if $\mu$ is an isomorphism.
\end{lem}
\begin{proof}
By \Cref{enrichedtrans} there is a $\VV$-enriched natural transformation 
\[\underline{\mu}:\underline{F}\to\underline{G}\]
that at each object $X$ evaluates to \[\underline{\mu}_X:1\to\uD(FX,GX)\] defined to be the adjoint of $\mu_X:FX\to GX$. The map $\uD(GX,Y)\to\uD(FX,Y)$ is defined as the composite

\begin{equation}\label{enrichedmap}
\uD(GX,Y)\cong\uD(GX,Y)\otimes 1\xrightarrow{1\otimes\umu_X}\uD(GX,Y)\otimes\uD(FX,GX)\xrightarrow{c}\uD(FX,Y),
\end{equation}
where $c$ is the composition map in the enriched setting. 

When $\mu$ is an isomorphism we may analogously define the following map

\begin{equation}\label{enrichedmapinverse}
\uD(FX,Y)\cong\uD(FX,Y)\otimes 1\xrightarrow{1\otimes\umui_X}\uD(FX,Y)\otimes\uD(GX,FX)\xrightarrow{c}\uD(GX,Y).
\end{equation}

We show that this map is the inverse of the map in \Cref{enrichedmap}.

\begin{equation}\label{complicateddiagram}
\begin{tikzcd}[column sep = 0pt, row sep = 20pt]
{\uD(GX,Y) } \arrow[r, "\cong"] \arrow[rd, "(5)", phantom, bend left = 7]                 & {\uD(GX,Y)\otimes 1} \arrow[r, "1\otimes\umu_X"] \arrow[d, "1\otimes\alpha_X"] \arrow[rd, "(4)", phantom] & {\uD(GX,Y)\otimes\uD(FX,GX)} \arrow[r, "c"] \arrow[d, "\cong"]                                                                     & {\uD(FX,Y)} \arrow[ddd, "\cong"]                                               \\
                                                                                     & {\uD(GX,Y)\otimes\uD(GX,GX)} \arrow[lu, "c"] \arrow[lu]                                                  & {\uD(GX,Y)\otimes\uD(FX,GX)\otimes 1} \arrow[ld, "1\otimes 1\otimes \umui_X"] \arrow[rdd, "c\otimes 1"] \arrow[ru, "(1)"', phantom] &                                                                                \\
                                                                                     & {\uD(GX,Y)\otimes\uD(FX,GX)\otimes \uD(GX,FX)} \arrow[u, "1\otimes c"] \arrow[ld, "c\otimes 1"]          &                                                                                                                                    &                                                                                \\
{\uD(FX,Y)\otimes\uD(GX,FX)} \arrow[uuu, "c"] \arrow[ruu, "(3)"', phantom, bend left] & {}                                                                                                       &                                                                                                                                    & {\uD(FX,Y)\otimes 1} \arrow[lll, "1\otimes\umui_X"] \arrow[llu, "(2)"', phantom]
\end{tikzcd}
\end{equation}

In the above diagram (\ref{complicateddiagram}), $\alpha_X$ is adjoint to $1_{GX}:GX\to GX$. Diagrams (1) and (2) clearly commute. Diagram (3) commutes by associativity of $c$. Diagram (4) commutes because $\umui_X$ and $\umu_X$ are adjoint to mutual inverses, so their composition results in the adjoint of the identity. Finally, diagram (5) commutes because we are composing with an isomorphism. In particular, diagram (5) is a decomposition of the identity map on $\uD(GX,Y)$. By commutativity, this means that the overall diagram composes to the identity, showing that the maps (\ref{enrichedmap}) and (\ref{enrichedmapinverse}) are mutually inverse.
\end{proof}

\begin{remark}\label{4.15}
The category $\fc$ is monoidal over $\vbc$. By restriction, $\fmod$ is monoidal over $\bgmod$.
\end{remark}

Next, we define some more essentials for our work with enriched categories. 
\begin{defin}\label{weirdenrichment}
Let $A,B,C$ be bigraded modules. We denote by $\underline{\mathpzc{bgMod}_R}(A,B)$ the bigraded module given by
\[\underline{\mathpzc{bgMod}_R}(A,B)^v_u :=\prod_{j\geq0}[A,B]^{v-j}_{u+j}\]
where $[A,B]$ is the internal hom. More precisely, $g \in\underline{\mathpzc{bgMod}_R}(A,B)^v_u$ is given
by the sequence $g := (g_0, g_1, g_2, \dots )$, where $g_j : A \longrightarrow B$ is a map of bigraded modules of bidegree $(u + j, v - j)$.
Moreover, we define a composition morphism
\[c : \underline{\mathpzc{bgMod}_R}(B,C) \otimes \underline{\mathpzc{bgMod}_R}(A,B) \longrightarrow \underline{\mathpzc{bgMod}_R}(A,C)\,\,\,\mbox{by}\,\,\,
c(f, g)_m :=\sum_{i+j=m}(-1)^{i|g|}f_ig_j .\]
\end{defin}

\begin{defin}\label{delta2}
Let $(A, d^A_i), (B, d^B_i)$ be twisted complexes, $f \in\underline{\mathpzc{bgMod}_R}(A,B)^v_u$ and consider $d^A :=(d^A_i)_i \in\underline{\mathpzc{bgMod}_R}(A,A)^1_0$
and $d^B := (d^B_i)_i \in\underline{\mathpzc{bgMod}_R}(B,B)^1_0$. We define
\[\delta (f) := c(d^B, f) - (-1)^{\langle f,d^A\rangle}c(f, d^A) \in\underline{\mathpzc{bgMod}_R}(A,B)^{v+1}_u.\]

More precisely,
\[(\delta (f))_m :=\sum_{i+j=m}(-1)^{i|f|}d^B_if_j - (-1)^{v+i}f_id^A_j.\]
\end{defin}

The following lemma justifies the above definition. 

\begin{lem}
The following equations hold.
\begin{align*}
c(d^A, d^A) = 0, \,\,\,\,\,
\delta ^2 = 0, \,\,\,\,\,\mbox{and}\,\,\,\,\,
\delta (c(f, g)) = c(\delta (f), g) + (-1)^v c(f, \delta (g))
\end{align*}
where $v$ is the vertical degree of $f$. Furthermore, $f \in\ubgMod(A,B)$ is a map of twisted complexes if and
only if $\delta (f) = 0$. In particular, $f$ is a morphism in $\tc$ if and only if the bidegree of $f$ is $(0, 0)$ and
$\delta (f) = 0$. Moreover, for $f$, $g$ morphisms in $\tc$, we have that $c(f, g) = f\circ g$, where the latter denotes
composition in $\tc$.
\end{lem}

\begin{defin}
For $A,B$ twisted complexes, we define $\underline{t\mathcal{C}_R}(A,B)$ to be the vertical bicomplex
$\underline{t\mathcal{C}_R}(A,B) := (\underline{\mathpzc{bgMod}_R}(A,B), \delta )$.
\end{defin}

\begin{defin}\label{ubgMod}
We denote by $\ubgMod$ the \emph{$\bgmod$-enriched category of bigraded modules} given
by the following data.

\begin{itemize}
\item The objects of $\ubgMod$ are bigraded modules.
\item For $A,B$ bigraded modules the hom-object is the bigraded module $\ubgMod(A,B)$.
\item The composition morphism $c:\ubgMod(B,C) \otimes \ubgMod(A,B) \longrightarrow \ubgMod(A,C)$ is given by \Cref{weirdenrichment}.
\item The unit morphism $R \longrightarrow \ubgMod(A,A)$ is given by the morphism of bigraded modules that
sends $1 \in R$ to $1_A : A \longrightarrow A$, the strict morphism given by the identity of $A$.
\end{itemize}
\end{defin}

\begin{defin}\label{utC}
The \emph{$\vbc$-enriched category of twisted complexes} $\utC$ is the enriched category given by the following data.

\begin{itemize}
\item The objects of $\utC$ are twisted complexes.
\item For $A,B$ twisted complexes the hom-object is the vertical bicomplex $\utC(A,B)$.
\item The composition morphism $c : \utC(B,C)\otimes\utC(A,B) \longrightarrow \utC(A,C)$ is given by \Cref{weirdenrichment}.
\item The unit morphism $R \longrightarrow \utC(A,A)$ is given by the morphism of vertical bicomplexes sending
$1 \in R$ to $1_A : A \longrightarrow A$, the strict morphism of twisted complexes given by the identity of $A$.
\end{itemize}
\end{defin}

The next tensor corresponds to $\underline{\otimes}$ in the categorical setting of \Cref{underline}.

\begin{lem}\label{tensorenriched}
The monoidal structure of $\utC$ is given by the following map of vertical bicomplexes.
\[\underline{\otimes}: \utC(A,B) \times \utC(A',B') \longrightarrow \utC(A \otimes A',B \otimes B'),\,\,\,(f, g) \mapsto (f\underline{\otimes}g)_m :=\sum_{i+j=m}(-1)^{ij}f_i \otimes g_j.\]
The monoidal structure of $\ubgMod$ is given by the restriction of this map.
\end{lem}

\begin{defin}\label{ufMod}
The \emph{$\bgmod$-enriched category of filtered modules} $\ufMod$ is the enriched category given by the following data.

\begin{itemize}
\item The objects of $\ufMod$ are filtered modules.
\item For filtered modules $(K, F)$ and $(L, F)$, the bigraded module $\ufMod(K,L)$ is given by
\[\ufMod(K,L)^v_u :=\{f : K \longrightarrow L\mid f(F_qK^m) \subset F_{q+u}L^{m+u+v}, \forall m, q \in\Z\}.\]
\item The composition morphism is given by $c(f, g) = (-1)^{u|g|}fg$, where $f$ has bidegree $(u, v)$.
\item The unit morphism is given by the map $R \longrightarrow \ufMod(K,K)$ given by $1 \longrightarrow 1_K$.
\end{itemize}
\end{defin}

\begin{defin}\label{fmoddifferential}
Let $(K, d^K, F)$ and $(L, d^L, F)$ be filtered complexes. We define $\ufC(K,L)$ to be the
vertical bicomplex whose underlying bigraded module is $\ufMod(K,L)$ with vertical differential
\[\delta (f) := c(d^L, f) - (-1)^{\langle f,d^K\rangle}c(f, d^K) = d^Lf - (-1)^{v+u}fd^K = d^Lf - (-1)^{|f|}fd^K\]
for $f \in\ufMod(K,L)^v_u$, where $c$ is the composition map from \Cref{ufMod}.
\end{defin}

\begin{defin}\label{ufC}
The \emph{$\vbc$-enriched category of filtered complexes} $\ufC$ is the enriched category given
by the following data.
\begin{itemize}
\item The objects of $\ufC$ are filtered complexes.
\item For $K,L$ filtered complexes the hom-object is the vertical bicomplex $\ufC(K,L)$.
\item The composition morphism is given as in $\ufMod$ in \Cref{ufMod}. 
\item The unit morphism is given by the map $R \longrightarrow \ufC(K,K)$ given by $1 \longrightarrow 1_K$.
We denote by $\usfC$ the full subcategory of $\ufC$ whose objects are split filtered complexes.

\end{itemize}
\end{defin}

The enriched monoidal structure is given as follows.
\begin{defin}\label{tensorenriched2}
The monoidal structure of $\ufC$ is given by the following map of vertical bicomplexes.
\[\underline{\otimes}: \ufC(K,L) \otimes \ufC(K',L') \longrightarrow \ufC(K \otimes K',L \otimes L'),\]
\[(f, g) \mapsto f\underline{\otimes}g := (-1)^{u|g|}f \otimes g\]
Here, $u$ is the horizontal degree of $f$.
\end{defin}

\begin{lem}\label{adjunction}
Let $A$ be a vertical bicomplex that is horizontally bounded on the right and let $K$ and $L$ be filtered complexes. There is a natural bijection
\[\Hom_{\fc}(\Tot(A)\otimes K,L)\cong \Hom_{\vbc}(A,\ufC(K,L))\]
given by $f\mapsto \tilde{f}: a\mapsto (k\mapsto f(a\otimes k))$.
\end{lem}

We now define an enriched version of the totalization functor. 
\begin{defin}\label{enrichedtot}
Let $A,B$ be bigraded modules and $f \in\ubgMod (A,B)^v_u$ we define

\[\Tot(f) \in\ufMod(\Tot(A),\Tot(B))^v_u\]
to be given on any $x \in\Tot(A)^n$ by
\[(\Tot(f)(x)))_{j+u} :=
\sum_{m\geq0}(-1)^{(m+u)n}f_m(x_{j-m}) \in B^{n-j+v}_{j+u} \subset \Tot(B)^{n+u+v}.\]
Let $K = \Tot(A)$, $L = \Tot(B)$ and $g \in\ufMod(K,L)^v_u$. We define
\[f := \Tot^{-1}(g) \in\ubgMod(A,B)^v_u\]
to be $f := (f_0, f_1,\dots)$ where $f_i$ is given on each $A^{m+j}_j$ by the composite
\begin{align*}
f_i : A^{m-j}_j \hookrightarrow\prod_{k\geq j}A^{m-k}_k &= F_j(\Tot(A)^m)\xrightarrow{g}F_{j+u}(\Tot(B)^{m+u+v})=\prod_{l\geq j+u}B^{m+u+v-l}_l\xrightarrow{\times(-1)^{(i+u)m}} B^{m-j+v-i}_{j+u+i} ,
\end{align*}
where the last map is a projection and multiplication with the indicated sign.
\end{defin} 

\begin{thm}[Cirici-Egas Santander-Livernet-Whitehouse]\label{4.39}
Let $A$, $B$ be twisted complexes. The assignments $\mathfrak{Tot}(A) := \Tot(A)$ and
\[
\mathfrak{Tot}_{A,B} : \utC(A,B) \longrightarrow \ufC(\Tot(A),\Tot(B)),\,  f \mapsto \Tot(f)
\]
define a $\vbc$-enriched functor $\mathfrak{Tot} : \utC \longrightarrow \ufC$ which restricts to an isomorphism onto its image. Furthermore, this functor restricts to a $\bgmod$-enriched functor \[\mathfrak{Tot} : \ubgMod \longrightarrow \ufMod\]
 which also restricts to an isomorphism onto its image.
\end{thm}

We now define an enriched endomorphism operad.
\begin{defin}
Let $\underline{\mathscr{C}}$ be a monoidal $\mathscr{V}$-enriched category and $A$ an object of $\uC$. We define $\uEnd_A$
to be the collection in $\mathscr{V}$ given by
\[\uEnd_A(n) \coloneqq \uC (A^{\otimes n},A) \text{ for }n \geq 1.\]
\end{defin}

\begin{propo}\label{S4}\
\begin{itemize}
\item The enriched functors 
\[\mathfrak{Tot} : \ubgMod  \longrightarrow \ufMod ,\hspace{1cm} \mathfrak{Tot} : \utC \longrightarrow \ufC\]
are lax symmetric monoidal in the enriched sense and when restricted to the bounded case they are strong symmetric monoidal in the enriched sense.
\item For $A\in\uC$, the collection $\uEnd_A$ defines an operad in $\VV$. 

\item Let $\CC$ and $\DD$ be monoidal categories over $\VV$. Let 
$F : \CC \longrightarrow \DD$ be a lax monoidal functor over $\VV$. Then for any $X \in\CC$ there is an operad morphism
\[\uEnd_X\longrightarrow\uEnd_{F(X)}.\]

\end{itemize}
\end{propo}

The following is an analog of \cite[Lemma 4.54]{whitehouse}, using \Cref{4.39}, \Cref{S4} and \Cref{4.15}.

\begin{lem}\label{inverse}
Let $A$ be a twisted complex. Consider $\uEnd_A(n)=\utC(A^{\otimes n},A)$ and $\uEnd_{\Tot(A)}(n)=\ufC(\Tot(A)^{\otimes n},\Tot(A))$. There is a morphism of operads
\[\uEnd_A \longrightarrow\uEnd_{\Tot(A)},\]
which is an isomorphism of operads if $A$ is bounded. The same holds true if $A$ is just a bigraded module. In that case, we use the enriched operads $\uEnd_A(n)=\ubgMod(A^{\otimes n},A)$ and $\uEnd_{\Tot(A)}(n)=\ufMod(\Tot(A)^{\otimes n},\Tot(A))$. 
\end{lem}


We are going to construct the inverse in the bounded case explicitly from \Cref{enrichedmap}. The construction for the direct map is analogue but here we just need the inverse. We do it for a twisted complex $A$, but it is done similarly for a bigraded module.

\begin{lem}\label{composition}
In the conditions of \Cref{inverse} for the bounded case, the inverse is given by the map
\[
\uEnd_{\Tot(A)}\to\uEnd_A,\, f  \mapsto \Tot^{-1}(f\circ \mu^{-1}).
\]
\end{lem}
\begin{proof}
The inverse is given by the composite
\[
\uEnd_{\Tot(A)}(n)=\ufC(\Tot(A)^{\otimes n},\Tot(A))\longrightarrow  \ufC(\Tot(A^{\otimes n}),\Tot(A))\longrightarrow\utC(A^{\otimes n},A)=\uEnd_A(n).
 \]

The second map is given by $\mathfrak{Tot}^{-1}$ defined in \Cref{enrichedtot}. To describe the first map, let $R$ be concentrated in bidegree $(0,0)$ with trivial vertical differential. Then the first map is given by the following composite
\begin{align*}
\ufC(\Tot(A)^{\otimes n},\Tot(A))\cong R\otimes\ufC(\Tot(A)^{\otimes n},\Tot(A))\xrightarrow{\underline{\mu}^{-1}\otimes 1}\\
\ufC(\Tot(A^{\otimes n}),\Tot(A)^{\otimes n})\otimes\ufC(\Tot(A)^{\otimes n},\Tot(A))
\xrightarrow{c}\ufC(\Tot(A^{\otimes n}),\Tot(A)), 
\end{align*}
where $c$ is the composition in $\ufC$, defined in \Cref{ufMod}. The map $\underline{\mu}^{-1}$ is the adjoint of $\mu^{-1}$ under the bijection from \Cref{adjunction}. Explicitly,
\[
\underline{\mu}^{-1}:R \to \ufC(\Tot(A^{\otimes n}),\Tot(A)^{\otimes n}),\, 1 \mapsto (a\mapsto \mu^{-1}(a)).
\]
Putting all this together, we get the map 
\[
\uEnd_{\Tot(A)}\to\uEnd_A,\, f  \mapsto \Tot^{-1}(c(f, \mu^{-1})).
\]
Since the total degree of $\mu^{-1}$ is 0, composition reduces to $c(f,\mu^{-1})=f\circ \mu^{-1}$ and we get the desired map.
\end{proof}

\section{Derived $A_\infty$-algebras and filtered $A_\infty$-algebras}\label{sec:deriveddef}

We assume that the reader is familiar with the basic definitions of $A_\infty$-algebras, although we will also recall some conventions if necessary. 
In this section we recall some definitions and results for our work with derived $A_\infty$-algebras and present some new ways of interpreting them in terms of operads and collections.

After introducing derived $A_\infty$-algebras we will furthermore recall the notion of filtered $A_\infty$-algebra, since it will play a role in linking derived $A_\infty$-algebras abd $A_\infty$-algebras using totalization.
Let us jump right in with the definition, using the grading and sign conventions from \Cref{sec:bigraded}.
  \begin{defin}\label{def:dAdef}
  A {derived $A_\infty$-algebra} on a $(\Z,\Z)$-bigraded $R$-module $A$ consist of a family of $R$-linear maps 
\[m_{ij}:A^{\otimes j}\to A\]
of bidegree $(i,2-(i+j))$ for each $j\geq 1$, $i\geq 0$, satisfying the equation
\begin{equation}\label{dainftyequation}
\underset{j=r+1+t}{\sum_{\mathclap{u=i+p, v=j+q-1}}}(-1)^{rq+t+pj}m_{ij}(1^{\otimes r}\otimes m_{pq}\otimes 1^{\otimes t})=0
\end{equation}
for all $u\geq 0$ and $v\geq 1$. 
\end{defin}

We therefore see that an $A_\infty$-algebra is the same as a derived $A_\infty$-algebra such that $m_{ij}=0$ for all $i>0$.
Furthermore, one can check that, on any derived $A_\infty$-algebra $A$, the maps $d_i=(-1)^{i}m_{i1}$ define a twisted complex structure. This leads to the possibility of defining a derived $A_\infty$-algebra as a twisted complex with some extra structure, see \Cref{equivalent}.

According to \Cref{def:dAdef}, there are two equivalent ways of defining the operad of derived $A_\infty$-algebras $d\calA_\infty$ depending on the underlying category. One of them works on the category of bigraded modules $\bgmod$ and the other one is suitable for the category of vertical bicomplexes $\vbc$. We give the two of them here as we are going to use both.

\begin{defin}
The operad $d\calA_\infty$ in $\bgmod$ is the operad generated by $\{m_{ij}\}_{i\geq 0,j\geq 1}$ subject to the derived $A_\infty$-relation

\[\underset{j=r+1+t}{\sum_{\mathclap{u=i+p, v=j+q-1}}}(-1)^{rq+t+pj}\gamma(m_{ij};1^{ r}, m_{pq}, 1^{t})=0\]
for all $u\geq 0$ and $v\geq 1$. 

The operad $d\calA_\infty$ in $\vbc$ is the quasi-free operad generated by $\{m_{ij}\}_{(i,j)\neq (0,1)}$ with vertical differential given by
\[\partial_\infty(m_{uv})=\ -\underset{\mathclap{j=r+1+t, (i,j)\neq (0,1)\neq (p,q)}}{\sum_{u=i+p, v=j+q-1}}(-1)^{rq+t+pj}\gamma(m_{ij};1^{ r}, m_{pq}, 1^{t}).\]
\end{defin}

\begin{defin}
Let $A$ and $B$ be derived $A_\infty$-algebras with respective structure maps $m^A$ and $m^B$. An \emph{$\infty$-morphism of derived $A_\infty$-algebras} $f:A\to B$ is a family of maps $f_{st}:A^{\otimes t}\to B$ of bidegree $(s,1-s-t)$ satisfying
\begin{equation}\label{dinftymaps}
\underset{j=r+1+t}{\sum_{u=i+p, v=j+q-1}}(-1)^{rq+t+pj}f_{ij}(1^{\otimes r}\otimes m_{pq}^A\otimes 1^{\otimes s})=\underset{v=q_1+\cdots +q_j}{\sum_{u=i+p_1+\cdots +p_j}}(-1)^{\epsilon} m^B_{ij}(f_{p_1 q_1}\otimes\cdots\otimes f_{p_j q_j})
\end{equation}
for all $u\geq 0$ and $v\geq 1$, where
\[\epsilon = u + \sum_{1\leq w < l \leq j} q_w(1-p_l-q_l)  + \sum_{w=1}^j p_w(j-w).\]
\end{defin}

We will make use of the filtration induced by the totalization functor in order to relate classical $A_\infty$-algebras to derived $A_\infty$-algebras.

\begin{defin}
A \emph{filtered} $A_\infty$-algebra is an $A_\infty$-algebra $(A,m_i)$ together with a filtration $\{F_pA^i\}_{p\in\Z}$
on each $R$-module $A^i$ such that for all $i \geq 1$ and all $p_1,\dots , p_i \in\Z$ and $n_1,\dots , n_i \geq 0$,
\[m_i(F_{p_1}A^{n_1} \otimes \cdots \otimes F_{p_i}A^{n_i} ) \subseteq F_{p_1+\cdots
+p_i}A^{n_1+\cdots+n_i+2-i}.\]
\end{defin}

\begin{remark}\label{filterversion}
Consider $\calA_\infty$ as an operad in filtered complexes with the trivial filtration and let $K$
be a filtered complex. There is a one-to-one correspondence between filtered $A_\infty$-algebra structures on $K$ and
morphisms of operads in filtered complexes $\calA_\infty \longrightarrow \underline{\End}_K$ (recall $\underline{\Hom}$ from \Cref{filterend}). To see this, notice that if one forgets the
filtrations, such a map of operads gives an $A_\infty$-algebra structure on $K$. The fact that this is a map of operads
in filtered complexes implies that all the $m_i$ respect the filtrations. 

The image of $\calA_\infty$ lies in $\End_K=F_0\underline{\End}_K$, so if we regard $\calA_\infty$ as an operad in cochain complexes, then we get a one-to-one correspondence between filtered $A_\infty$-algebra structures on $K$ and
morphisms of operads in cochain complexes $\calA_\infty \longrightarrow \End_K$.
\end{remark}

\begin{defin}
A \emph{morphism of filtered $A_\infty$-algebras} from $(A,m_i, F)$ to $(B,m_i, F)$ is an $\infty$-morphism
$f : (A,m_i) \longrightarrow (B,m_i)$ of $A_\infty$-algebras such that each map $f_j : A^{\otimes j} \longrightarrow A$ is compatible with filtrations, i.e.
\[f_j(F_{p_1}A^{n_1} \otimes \cdots \otimes F_{p_j}A^{n_j} ) \subseteq F_{p_1+\cdots +p_j}B^{n_1+\cdots +n_j+1-j} ,\]
for all $j \geq 1$, $p_1,\dots p_j \in\Z$ and $n_1,\dots , n_j \geq 0$.
\end{defin}

\section{Operadic totalization and vertical operadic suspension}\label{sec:operadic}

In this section we define an operadic suspension, which is a slight modification of the one found in \cite{ward}. This construction will help us define $A_\infty$-multiplications and derived $A_\infty$-multiplications in a simple way. The motivation to introduce operadic suspension is that signs in derived $A_\infty$-algebras (as well as the single graded setting) and related Lie structures are know to arise from a sequence of shifts. We are going to work only with non-symmetric operads, although most of what we do is also valid in the symmetric case.

We start by applying the totalization functor defined in \Cref{sec:total} to operads, defining a functor from operads in brigraded modules (resp. twisted complexes) to operads in graded modules (resp. cochain complexes). The combination of this with operadic suspension provides the signs required to encode derived $A_\infty$-algebras in a very concise and practical way.

 We use \Cref{monoidal} and the fact that the image of an operad under a lax monoidal functor is also an operad \cite[Proposition 3.1.1(a)]{fresse} to guarantee that applying totalization on an operad will result again in an operad.

Let $\OO$ be either a bigraded operad, i.e. an operad in te category of bigraded $R$-modules or an operad in twisted complexes. We define $\Tot(\OO)$ as the operad of graded $R$-modules (or cochain complexes) for which \[\Tot(\OO(n))^d=\bigoplus_{i<0}\OO(n)^{d-i}_i\oplus\prod_{i\geq 0} \OO(n)^{d-i}_i\] is the image of $\OO(n)$ under the totalization functor, and the insertion maps $\bar{\circ}_r$ are given by the composition  

\begin{equation}\label{insertion}
\Tot(\OO(n))\otimes \Tot(\OO(m))\xrightarrow{\mu} \Tot(\OO(n)\otimes \OO(m)) \xrightarrow{\Tot(\circ_r)} \Tot(\OO(n+m-1)).
\end{equation}
Explicitly,
\[(x\bar{\circ}_ry)_k=\sum_{k_1+k_2=k} (-1)^{k_1d_2} x_{k_1}\circ_r y_{k_2}\]

for $x=(x_i)_i\in \Tot(\OO(n))^{d_1}$ and $y=(y_j)_j\in \Tot(\OO(m))^{d_2}$.

More generally, operadic composition $\bar{\gamma}$ is defined by the composite
\[
\Tot(\OO(N))\otimes \Tot(\OO(a_1))\otimes\cdots\otimes \Tot(\OO(a_N)) \xrightarrow{\mu}
 \Tot(\OO(N)\otimes \OO(a_1)\otimes\cdots\otimes \OO(a_N)) \xrightarrow{\Tot(\gamma)}  \Tot\left(\OO\left(\sum a_i\right)\right).
\]
This map can be computed explicitly by iteration of the insertions, giving the following.  Note that the sign is precisely the same appearing in \Cref{mu}.

\begin{lem}\label{totcomp}
The operadic composition $\bar{\gamma}$ on $\Tot(\OO)$ is given by
\begin{equation*}
\bar{\gamma}(x;x^1,\dots, x^N)_k=\sum_{k_0+k_1+\cdots+k_N=k}(-1)^{\varepsilon}\gamma(x_{k_0};x^1_{k_1},\dots, x^N_{k_N})
\end{equation*}
for $x=(x_k)_k\in\Tot(\OO(N))^{d_0}$ and $x^i=(x^i_k)_k\in\Tot(\OO(a_i))^{d_i}$, where 
\begin{equation}
\varepsilon=\sum_{j=1}^m d_j\sum_{i=0}^{j-1}k_i
\end{equation}
and $\gamma$ is the operadic composition on $\OO$.
\end{lem}

Let us now move on to defining operadic suspension for our setting.
We define $\Lambda(n)=S^{n-1}R$, where  $S$ is a vertical shift of degree so that $\Lambda(n)$ is the underlying ring $R$ concentrated in bidegree  $(0,n-1)$. We express the basis element of $\Lambda(n)$ as $e^n=e_1\land\cdots\land e_n$. We then have an operad structure on $\Lambda=\{\Lambda(n)\}_{n\geq 0}$ via the following insertion maps

\begin{equation}\label{lambdainsertion}
\begin{tikzcd}
\Lambda(n)\otimes\Lambda(m) \arrow[r, "\circ_i"] & \Lambda(n+m-1)\\
(e_1\land\cdots\land e_n)\otimes(e_1\land\cdots\land e_m)\arrow[r, mapsto] & (-1)^{(n-i)(m-1)}e_1\land\cdots\land e_{n+m-1}.
\end{tikzcd}
\end{equation}

We are inserting the second factor onto the first one, so the sign can be explained  by moving the power $e^m$ of degree $m-1$ to the $i$-th position of $e^n$ passing by $e_{n}$ through $e_{i+1}$. More compactly, \[e^n\circ_i e^m=(-1)^{(n-i)(m-1)}e^{n+m-1}.\] The unit of this operad is $e^1\in\Lambda(1)$. It can be checked by direct computation that $\Lambda$ satisfies the axioms of an operad of graded modules. In a similar way we can define $\Lambda^-(n)=S^{1-n}R$, with the same insertion maps.

\begin{defin}
Let $\mathcal{O}$ be a bigraded linear operad. The \emph{vertical operadic suspension} $\mathfrak{s}\OO$ of $\mathcal{O}$ is given arity-wise by $\mathfrak{s}\OO(n)=\mathcal{O}(n)\otimes\Lambda(n)$ with diagonal composition. Similarly, we define the \emph{vertical operadic desuspension} $\mathfrak{s}^{-1}\OO(n)=\mathcal{O}(n)\otimes\Lambda^-(n)$.
\end{defin}

We may identify the elements of $\mathcal{O}$ with the elements of $\mathfrak{s}\OO$ as follows.
\begin{defin}
For $x\in\OO(n)$ of bidegree $(k,d-k)$, its \emph{natural bidegree} in $\s\OO$ is the pair $(k,d+n-k-1)$. To distinguish both degrees we call $(k,d-k)$ the \emph{internal bidegree} of $x$, since this is the degree that $x$ inherits from the grading of $\OO$. 
\end{defin}

If we write $\circ_{r+1}$ for the operadic insertion on $\OO$ and $\tilde{\circ}_{r+1}$ for the operadic insertion on $\mathfrak{s}\OO$, we may find a relation between the two insertion maps. 

\begin{lem}\label{sign}
For $x\in\OO(n)$ and $y\in\OO(m)^{q}_l$ we have $
x\tilde{\circ}_{r+1}y=(-1)^{(n-1)q+(n-1)(m-1)+r(m-1)}x\circ_{r+1} y$
\qed
\end{lem}

\begin{remark}
This operation leads to the Lie bracket from \cite{RW}, which implies that $m=\sum_{i,j}m_{ij}$ is a derived $A_\infty$-multiplication if and only if for all $u\geq 0$
\begin{equation}\label{sharp}
\sum_{i+j=u}\sum_{l,k}(-1)^im_{jl}\tilde{\circ}m_{ik}=0.
\end{equation}
In \cite[Proposition 2.15]{RW} this equation is described in terms of a sharp operator $\sharp$.
\end{remark}

We then arrive at the following theorem, which is a generalisation of \cite[Chapter 3, Lemma 3.16]{operads}. The original statement is about vector spaces, but it is still true when $R$ is not a field. The isomorphism is given by $\sigma^{-1}(F)=(-1)^{\binom{n}{2}}S^{-1}\circ F\circ S^{\otimes n}$ for $F\in \End_{S A}(n)$. The symbol $\circ$ here is just composition.

\begin{thm}\label{thm:shift}
There is an isomorphism of operads $\End_{ A}\cong \mathfrak{s}\End_{SA}$ for any bigraded $R$-module $A$.\qed
\end{thm}

Even though $\sigma$ is only a map of graded modules, it can be shown in a completely analogous way to the above theorem that $\bar{\sigma}=(-1)^{\binom{n}{2}}\sigma$ induces an isomorphism of operads

\begin{equation}\label{eq:barsigma}
\bar{\sigma} : \End_A \cong \mathfrak{s}\End_{SA}.
\end{equation}

\begin{remark}
The functor $\s:\col\to \col$ defines a lax monoidal functor. When restricted to the subcategory of reduced operads, it is strong monoidal. This can be verified straight from the definitions given here, being mindful of the signs introduced by the Koszul rule.
\end{remark}

Now we are going to combine vertical operadic suspension and totalization. More precisely, the \emph{totalized vertical suspension} of a bigraded operad $\OO$ is the graded operad $\Tot(\s\OO)$. This operad has an insertion map explicitly given by
\begin{equation}\label{star}
(x\star_{r+1} y)_k=\sum_{k_1+k_2=k}(-1)^{(n-1)(d_2-k_2-m+1)+(n-1)(m-1)+r(m-1)+k_1d_2}x_{k_1}\circ_{r+1}y_{k_2}
\end{equation}
for $x=(x_i)_i\in \Tot(\s\OO(n))^{d_1}$ and $x=(x_j)_j\in \Tot(\s\OO(m))^{d_2}$. As usual, denote \[x\star y=\sum_{r=0}^{m-1}x\star_{r+1}y.\]

This star operation is precisely the star operation from \cite[\S 5.1]{LRW}, i.e. the convolution operation on $\Hom((dAs)^{!}, \End_A)$. In particular, we can recover the Lie bracket from in \cite{LRW}. We will do this in \Cref{biliebracket}.

We note that of course if we work in bigraded modules concentrated in horizontal degree 0, we recover the classical notion of (single-graded) operadic suspension, so, before continuing, let us show a lemma that allows us to work only with the single-graded operadic suspension if needed.
\begin{propo}\label{extrasign}
For a bigraded operad $\OO$ we have an isomorphism $\Tot(\s\OO)\cong \s \Tot(\OO)$, where the suspension on the left hand side is the bigraded version and on the right hand side is the single-graded version. 
\end{propo}
\begin{proof}
 Note that we may identify each element $x=(x_k\otimes e^n )_k\in\Tot(\s\OO(n))$ with the element $x=(x_k)_k\otimes e^n\in\s\Tot(\OO(n))$. Thus, for an element $(x_k)_k\in \Tot(\s\OO(n))$ the isomorphism is given by
\[
f:\Tot(\s\OO(n))\cong \s \Tot(\OO(n)),\, (x_k)_k\mapsto ((-1)^{kn}x_k)_k
\]
Clearly, this map is bijective so we just need to check that it commutes with insertions. Recall from \Cref{star} that the insertion on $\Tot(\s\OO)$ is given by
\begin{equation*}
(x\star_{r+1} y)_k=\sum_{k_1+k_2=k}(-1)^{(n-1)(d_2-k_2-n+1)+(n-1)(m-1)+r(m-1)+k_1d_2}x_{k_1}\circ_{r+1}y_{k_2}
\end{equation*}
for $x=(x_i)_i\in \Tot(\s\OO(n))^{d_1}$ and $y=(y_j)_j\in \Tot(\s\OO(m))^{d_2}$. Similarly, we may compute the insertion on $\s\Tot(\OO)$ by combining the sign produced first by $\Tot$ and then by $\s$. This results in the following insertion map 
\begin{equation*}
(x\star_{r+1}' y)_k=\sum_{k_1+k_2=k}(-1)^{(n-1)(d_2-n+1)+(n-1)(m-1)+r(m-1)+k_1(d_2-m+1)}x_{k_1}\circ_{r+1}y_{k_2}.
\end{equation*}
Now let us show that $f(x\star y)=f(x)\star f(y)$. We have that $f((x\star_{r+1} y))_k$ equals 
\begin{align*}
&\sum_{k_1+k_2=k}(-1)^{k(n+m-1)+(n-1)(d_2-k_2-n+1)+(n-1)(m-1)+r(m-1)+k_1d_2}x_{k_1}\circ_{r+1}y_{k_2}\\
&=\sum_{k_1+k_2=k}(-1)^{(n-1)(d_2-n+1)+(n-1)(m-1)+r(m-1)+k_1(d_2-m+1)}f(x_{k_1})\circ_{r+1}f(y_{k_2})\\
&=(f(x)\star_{r+1} f(y))_k
\end{align*}
as desired.
\end{proof}

\begin{remark}\label{othermu}

As we mentioned in \Cref{heuristic}, there exist other possible ways of totalizing by varying the natural transformation $\mu$. For instance, we can choose the totalization functor $\Tot'$ which is the same as $\Tot$ but with a natural transformation $\mu'$ defined in such a way that the insertion on $\Tot'(\OO)$ is defined by \[(x\hat{\circ}y)_k=\sum_{k_1+k_2=k}(-1)^{k_2n_1}x_{k_1}\circ y_{k_2}.\] 

This is also a valid approach for our purposes and there is simply a sign difference, but we have chosen our convention to be consistent with other conventions, such as the derived $A_\infty$-equation. However, it can be verified that $\Tot'(\s\OO)=\s \Tot'(\OO)$. With the original totalization we have a non identity isomorphism given by \Cref{extrasign}. Similar relations can be found among the other alternatives mentioned in \Cref{heuristic}.

\end{remark}

Using the operadic structure on $\Tot(\s\OO)$, we can describe now derived $A_\infty$-multiplications.

\begin{defin}\label{derivedmultiplication}
A \emph{derived $A_\infty$-multiplication} on a bigraded operad $\OO$ is a map of operads $d\calA_\infty\to\OO$.
\end{defin}

In particular, a derived $A_\infty$-algebra $A$ is equivalent to a derived $A_\infty$-multiplication on its endomorphism operad.

\begin{lem}\label{mstar}
A derived $A_\infty$-multiplication on a bigraded operad $\OO$ is equivalent to an element $m\in\Tot(\s\OO)$ of degree 1 concentrated in positive arity such that $m\star m = 0$. 
\end{lem}

\begin{proof}

A derived $A_\infty$-multiplication on $\OO$ is by \Cref{derivedmultiplication} a map $f:d\calA_\infty\to\OO$.
Since $\calA_\infty$ is generated by elements $\mu_{ij}$ of bidegree $(i,2-i-j)$, such a map is determined by the elements $m_{ij}=f(\mu_{ij})\in\OO^{2-i-j}_i(j)$. Consider $m_j = (m_{ij})_i\in\Tot(\s\OO(j))$. We have that $\deg(m_j)=1$ for all $j$. Therefore, let $m=m_1+m_2+\cdots\in\Tot(\s\OO)$. We may check that $m\star m=0$. For that we just need to check \Cref{star}. On arity $n$, this amounts to computing
\[(m\star m)_k = \sum_{r=0}^{n-1}\underset{j+q=n-1}{\sum_{i+p=k}}(-1)^{rp+j-r-1+ pj}m_{ij}\circ_{r+1}m_{pq}=0.\]
The above expression vanishes precisely because the elements $m_{ij}$ satisfy the derived $A_\infty$-equation.

Conversely, let $m\in\Tot(\s\OO)$ of degree 1, is concentrated in positive arity and satisfying $m\star m=0$. We can split $m$ into its arity and horizontal degree components as $m=\sum_{i,j}m_{ij}$. As we have seen, the fact that $m\star m=0$ is equivalent to the elements $m_{ij}$ satisfying the derived $A_\infty$-equation, and therefore, a map $f:d\calA_\infty\to\OO$ is determined by $f(\mu_{ij})=m_{ij}$, which is of bidegree $(i,2-i-j)$. 
\end{proof}

\begin{remark}\label{rem:ainfinitymult}
Note that there are obvious analogous definitions for the less structured situations, i.e. 
an operad $\OO$ has an $A_\infty$-multiplication if there is a map $A_\infty \longrightarrow \OO$  from the operad $A_\infty$. 
An $A_\infty$-multiplication on an operad $\OO$ is equivalent to an element $m\in\s\OO$ of degree 1 concentrated in positive arity such that $m\tilde{\circ}m=0$, where $x\tilde{\circ} y=\sum_i x\tilde{\circ}_i y$, and $\tilde{\circ}_i$ is the operadic insertion in $\mathfrak{s}\OO$.
\end{remark}

\section{Bigraded braces and totalized braces}\label{sec:braces}

Brace algebras appear naturally in the context of operads when we fix the first argument of operadic composition \cite{GV}. This simple idea gives rise to a very rich structure that will be relevant to our work with derived $A_\infty$-structures. We will start off with recalling the classical definition of braces in (single) graded modules. Once we are familiar with these, we can move on to generalising the definition to the bigraded context. We chose this method of presentation as we will use single graded braces again when discussing the classical Deligne conjecture in \Cref{sec:classicaldeligne}, plus we find it easier to follow and develop an intuition for the material starting with the single graded definition.

\begin{defin}\label{braces}
A \emph{brace algebra} on a graded module $A$ consists of a family of maps \[b_n:A^{\otimes 1+n}\to A\] called \emph{braces}, that we evaluate on $(x,x_1,\dots, x_n)$ as $b_n(x;x_1,\dots, x_n)$. They must satisfy the \emph{brace relation}

\begin{align*}
b_m(b_n(x;x_1,\dots, x_n);y_1,\dots,y_m)=&\\
\underset{j_1\dots, j_n}{\sum_{i_1,\dots, i_n}}(-1)^{\varepsilon}b_l(x; y_1,\dots, y_{i_1},b_{j_1}(x_1;y_{i_1+1},&\dots, y_{i_1+j_1}),\dots, b_{j_n}(x_n;y_{i_n+1},\dots, y_{i_n+j_n}),\dots,y_m)
\end{align*}
where $l=n+\sum_{p=1}^n i_p$ and $\varepsilon=\sum_{p=1}^n\deg(x_p)\sum_{q=1}^{i_p}\deg(y_q),$ i.e. the sign is picked up by the $x_i$'s passing by the $y_i$'s in the shuffle.
\end{defin}

\begin{remark}
Some authors might use the notation $b_{1+n}$ instead of $b_n$, but the first element is usually going to have a different role from the others, so we found $b_n$ more intuitive. A shorter notation for $b_n(x;x_1,\dots,x_n)$ found in the literature (\cite{GV}, \cite{getzler}) is $x\{x_1,\dots, x_n\}$. 
\end{remark}

Operads naturally carry a brace algebra structure as follows. 
Given an operad $\OO$ with composition map $\gamma:\OO\circ\OO\to\OO$ we can define a brace algebra on the underlying module of $\OO$ by setting
\[b_n:\OO(N)\otimes\OO(a_1)\otimes\cdots\otimes\OO(a_n)\to\OO(N-n+\sum a_i)\]

\[b_n(x;x_1,\dots, x_n)=\sum\gamma(x;1,\dots,1,x_1,1,\dots,1,x_n,1,\dots,1),\]
where the sum runs over all possible order-preserving insertions. The brace $b_n(x;x_1,\dots,x_n)$ vanishes whenever $n>N$ and $b_0(x)=x$. The brace relation follows from the associativity axiom of operads.

This construction can  be used to define braces on the suspension $\s\OO$. More precisely, we define maps 
\[b_n:\mathfrak{s}\OO(N)\otimes\mathfrak{s}\OO(a_1)\otimes\cdots\otimes\mathfrak{s}\OO(a_n)\to\mathfrak{s}\OO(N-n+\sum a_i)\]
using the operadic composition $\tilde{\gamma}$ on $\mathfrak{s}\OO$ as

\[b_n(x;x_1,\dots,x_n)=\sum\tilde{\gamma}(x;1,\dots,1,x_1,1,\dots,1,x_n,1,\dots,1).\]

We have the following relation between the brace maps $b_n$ defined on $\s\OO$ and the operadic composition $\gamma$ on $\OO$. As this precise sign is needed in several key places in this article, we provide an explicit calculation.

\begin{propo}\label{bracesign}
For $x\in \s\OO(N)$ and $x_i\in\s\OO(a_i)$ of internal degree $q_i$ ($1\leq i\leq n$), we have
\[b_n(x;x_1,\dots,x_n)=\sum_{N-n=k_0+\cdots+k_n} (-1)^\eta \gamma
(x\otimes 1^{\otimes k_0}\otimes x_1\otimes \cdots\otimes x_n\otimes1^{\otimes k_n}),\]
where 
\[\eta=\sum_{0\leq j<l\leq n}k_jq_l+\sum_{1\leq j<l\leq n}a_jq_l+\sum_{j=1}^n (a_j+q_j-1)(n-j)+\sum_{1\leq j\leq l\leq n} (a_j+q_j-1)k_l.\]
\end{propo}

\begin{proof}
To obtain the signs that make $\tilde{\gamma}$ differ from $\gamma$, we must first look at the operadic composition on $\Lambda$. 
We are interested in compositions of the form \[\tilde{\gamma}(x\otimes 1^{\otimes k_0}\otimes x_1\otimes 1^{\otimes k_1}\otimes\cdots\otimes x_n\otimes 1^{\otimes k_n})\] where $N-n=k_0+\cdots+k_n$, $x$ has arity $N$ and each $x_i$ has arity $a_i$ and internal degree $q_i$. Therefore, let us consider the corresponding operadic composition 

\[
\begin{tikzcd}
\Lambda(N)\otimes\Lambda(1)^{k_0}\otimes\Lambda(a_1)\otimes\Lambda(1)^{\otimes k_1}\otimes\cdots\otimes\Lambda(a_n)\otimes\Lambda(1)^{k_n}\arrow[r] & \Lambda(N-n+\sum_{i=1}^na_i).
\end{tikzcd}
\]

The operadic composition can be described in terms of insertions in the obvious way, namely, if $f\in\s\OO(N)$ and $h_1,\dots, h_N\in\s\OO$, then we have

\[\tilde{\gamma}(x;y_1,\dots, y_N)=(\cdots(x\tilde{\circ}_1 y_1)\tilde{\circ}_{1+a(y_1)}y_2\cdots)\tilde{\circ}_{1+\sum a(y_p)}y_N,\]

where $a(y_p)$ is the arity of $y_p$ (in this case $y_p$ is either $1$ or some $x_i$). So we just have to find out the sign iterating the same argument as in the $i$-th insertion. In this case, each $\Lambda(a_i)$ produces a sign given by the exponent $$(a_i-1)(N-k_0+\cdots-k_{i-1}-i).$$ 

For this, recall that the degree of $\Lambda(a_i)$ is $a_i-1$ and that the generator of this space is inserted in the position $1+\sum_{j=0}^{i-1}k_j+\sum_{j=1}^{i-1}a_j$ of a wedge of $N+\sum_{j=1}^{i-1}a_j-i+1$ generators. Therefore, performing this insertion as described in the previous section yields the aforementioned sign. Now, since $N-n=k_0+\cdots+k_n$, we have that
\[(a_i-1)(N-k_0+\cdots+k_{i-1}-i)=(a_i-1)(n-i+\sum_{l=i}^nk_l).\]

Now we can compute the sign factor of a brace. For this, notice that the isomorphism $(\OO(1)\otimes \Lambda(1))^{\otimes k}\cong \OO(1)^{\otimes k}\otimes \Lambda(1)^{\otimes k}$ does not produce any signs because of degree reasons. Therefore, the sign coming from the isomorphism

\begin{align*}
\OO(N)\otimes\Lambda(N)\otimes (\OO(1)\otimes \Lambda(1))^{\otimes k_0}\otimes \bigotimes_{i=1}^n(\OO(a_i)\otimes\Lambda(a_i)\otimes(\OO(1)\otimes\Lambda(1))^{\otimes k_i}\cong\\
 \OO(N)\otimes\OO(1)^{\otimes k_0}\otimes(\bigotimes_{i=1}^n \OO(a_i)\otimes \OO(1)^{\otimes k_i})\otimes \Lambda(N)\otimes\Lambda(1)^{\otimes k_0}\otimes(\bigotimes_{i=1}^n \Lambda(a_i)\otimes \Lambda(1)^{\otimes k_i})
\end{align*}
is determined by the exponent

\[(N-1)\sum_{i=1}^nq_i+\sum_{i=1}^n (a_i-1)\sum_{l>i}q_l.\]

This equals
\[\left(\sum_{j=0}^nk_j +n-1\right)\sum_{i=1}^nq_i+\sum_{i=1}^n (a_i-1)\sum_{l>i}q_l.\]

After doing the operadic composition 
\[\OO(N)\otimes(\bigotimes_{i=1}^n \OO(a_i))\otimes \Lambda(N)\otimes(\bigotimes_{i=1}^n \Lambda(a_i))\longrightarrow \OO(N-n+\sum_{i=1}^na_i)\otimes \Lambda(N-n+\sum_{i=1}^na_i)\]

we can add the sign coming from the suspension, so all in all the sign $(-1)^\eta$ we were looking for is given by

\[\eta=\sum_{i=1}^n(a_i-1)(n-i+\sum_{l=i}^nk_l)+(\sum_{j=0}^nk_j +n-1)\sum_{i=1}^nq_i+\sum_{i=1}^n (a_i-1)\sum_{l>i}q_l.\]

It can be checked that this can be rewritten modulo $2$ as 
\[\eta=\sum_{0\leq j<l\leq n}k_jq_l+\sum_{1\leq j<l\leq n}a_jq_l+\sum_{j=1}^n (a_j+q_j-1)(n-j)+\sum_{1\leq j\leq l\leq n} (a_j+q_j-1)k_l\]
as we stated.
\end{proof}

Note that for $\OO=\End_A$, the brace on operadic suspension from \Cref{bracesign} is precisely $[f,g]=b_1(f;g)-(-1)^{|f||g|}b_1(g;f)$ defined in \cite{RW}.

\bigskip
We are going to define a brace structure on $\Tot(\s\OO)$ using totalization. One defines bigraded braces just like in the single-graded case, only changing the sign $\varepsilon$ in \Cref{braces} to be $\varepsilon=\sum_{p=1}^n\sum_{q=i}^{i_p}\langle x_p,y_q\rangle$ according to the bigraded sign convention.

As one might expect, we can define bigraded brace maps $b_n$ on a bigraded operad $\OO$ and also on its operadic suspension $\s\OO$, obtaining similar signs as in the single-graded case, but with vertical (internal) degrees, see \Cref{bracesign}. 

We can also define braces on $\Tot(\s\OO)$ via operadic composition. In this case, these are usual single-graded braces. More precisely, we define the maps 
\[b^\star_n:\Tot(\mathfrak{s}\OO(N))\otimes \Tot(\mathfrak{s}\OO(a_1))\otimes\cdots\otimes \Tot(\mathfrak{s}\OO(a_n))\to \Tot(\mathfrak{s}\OO(N-\sum a_i))\]
using the operadic composition $\gamma^\star$ on $\Tot(\mathfrak{s}\OO)$ as

\[b^\star_n(x;x_1,\dots,x_n)=\sum\gamma^\star(x;1,\dots,1,x_1,1,\dots,1,x_n,1,\dots,1),\]

where the sum runs over all possible ordering preserving insertions. The brace map $b^\star_n(x;x_1,\dots,x_n)$ vanishes whenever $n>N$ and $b^\star_0(x)=x$. 

Operadic composition can be described in terms of insertions in the obvious way, namely 

\begin{equation}\label{gammastar}
\gamma^\star(x;y_1,\dots,y_N)=(\cdots(x\star_1 y_1)\star_{1+a(y_1)}y_2\cdots)\star_{1+\sum a(y_p)}y_N,
\end{equation}

where $a(y_p)$ is the arity of $y_p$. If we want to express this composition in terms of the composition in $\OO$ we just have to find out the sign factor applying the same strategy as in the single-graded case. In fact, as we said, there is a sign factor that comes from vertical operadic suspension that is identical to the graded case, but replacing internal degree by internal vertical degree. This is the sign that determines the brace $b_n$ on $\s\OO$. Explicitly, it is given by the following lemma, whose proof is identical to that of \Cref{bracesign}.

 \begin{lem}\label{bigradedsign}
For $x\in \s\OO(N)$ and $x_i\in\s\OO(a_i)$ of internal vertical degree $q_i$ ($1\leq i\leq n$), we have
\[b_n(x;x_1,\dots,x_n)=\sum_{N-n=h_0+\cdots+h_n} (-1)^\eta \gamma
(x\otimes 1^{\otimes h_0}\otimes x_1\otimes \cdots\otimes x_n\otimes1^{\otimes h_n}),\]
where 
\[\eta=\sum_{0\leq j<l\leq n}h_jq_l+\sum_{1\leq j<l\leq n}a_jq_l+\sum_{j=1}^n (a_j+q_j-1)(n-j)+\sum_{1\leq j\leq l\leq n} (a_j+q_j-1)h_l.\] \qed
\end{lem}

The other sign factor is produced by totalization, see \Cref{totcomp}. Combining both factors we obtain the following.

\begin{lem}
We have 
\begin{equation}\label{bracetot}
b_j^\star(x;x^1,\dots, x^N)_k=\underset{h_0+h_1+\cdots+h_N=j-N}{\sum_{k_0+k_1+\cdots+k_N=k}}(-1)^{\eta+\sum_{j=1}^m d_j\sum_{i=0}^{j-1}k_i}\gamma(x_{k_0};1^{h_0},x^1_{k_1},1^{h_1},\dots, x^N_{k_N},1^{h_N})
\end{equation}
for $x=(x_k)_k\in\Tot(\s\OO(N))^{d_0}$ and $x^i=(x^i_k)_k\in\Tot(\s\OO(a_i))^{d_i}$, where $\eta$ is defined in \Cref{bigradedsign}. \qed
\end{lem}

\begin{corollary}\label{biliebracket}
 For $\OO = \End_A$, the endomorphism operad of a bigraded module, the brace $b_1^\star(f;g)$ is the operation $f\star g$ defined in \cite{LRW} that induces a Lie bracket. More precisely,
\[
[f,g]=b_1(f;g)-(-1)^{NM}b_1(g;f)
\]
for $f\in\Tot(\s\End_A)^N$ and $g\in\Tot(\s\End_A)^M$, is the same bracket that was defined in \cite{LRW}. \qed
\end{corollary}

Note that in \cite{LRW} the sign in the bracket is $(-1)^{(N+1)(M+1)}$, but this is because their total degree differs by one with respect to ours.

\section{The Classical Deligne Conjecture}\label{sec:classicaldeligne}

In this section we use the previously described brace structures to give a rigorous proof of \Cref{theorem}, which was originally claimed by Gerstenhaber and Voronov \cite{GV}. This leads us to our first new version of the Deligne conjecture,  \Cref{ainftydeligne}.

Let $\OO$ be an operad of graded $R$-modules and $\s\OO$ its operadic suspension. Let us consider the underlying graded module of the operad $\s\OO$, which we call $\s\OO$ again by abuse of notation, i.e. $\s\OO=\prod_n \s\OO(n)$ with grading given by its natural degree  $|x|=n+\deg(x)-1$ for $x\in \s\OO(n)$, where $\deg(x)$ is its internal degree.

Recall from \Cref{rem:ainfinitymult} that an $A_\infty$ multiplication on an operad $\OO$ is a map of operads $A_\infty \longrightarrow \OO$, or, equivalently, an element $m \in \s\OO$ of degree 1 with $m \tilde{\circ} m =0$, where $\tilde{\circ}$ is the operadic insertion in $\s\OO$. 
One can attempt to define an $A_\infty$-algebra structure on $\s\OO$ following \cite{GV} and \cite{getzler} using the maps

\begin{align*}
M'_1(x)\coloneqq [m,x]=m\tilde{\circ} x-(-1)^{|x|}x\tilde{\circ}m, & &  \\
M'_j(x_1,\dots, x_j)\coloneqq b_j(m;x_1,\dots, x_j),& &j>1.
\end{align*}
The prime notation here is used to indicate that these are not the definitive maps that we are going to take. Getzler shows in \cite{getzler} that $M'=M'_1+M'_2+\cdots$ satisfies the relation $M'\circ M'=0$ using that $m\circ m=0$, and the proof is independent of the operad in which $m$ is defined, so it is still valid if $m\tilde{\circ}m=0$. But we have two problems here. The equation $M'\circ M'=0$ does depend on how the circle operation is defined. More precisely, this circle operation in \cite{getzler} is the natural circle operation on the endomorphism operad, which does not have any additional signs, so $M'$ is not an $A_\infty$-structure under our convention. The other problem has to do with the degrees. We need $M'_j$ to be homogeneous of degree $2-j$ as a map $\s\OO^{\otimes j}\to \s\OO$, but we find that $M'_j$ is homogeneous of degree 1 instead, as the following lemma shows.
\begin{lem}\label{lemmadegree}
For $x\in\s\OO$ we have that the degree of the map $b_j(x;-):\s\OO^{\otimes j}\to\s\OO$ of graded modules is precisely $|x|$.
\end{lem}

\begin{proof}
This is verified with a direct computation, using that 
the natural degree of $b_j(x;x_1,\dots,x_j)$ for $a(x)\geq j$ ($a(x)$ being the arity of $x$) as an element of $\s\OO$ by definition is

\[|b_j(x;x_1,\dots,x_j)|=a(b_j(x;x_1,\dots,x_j))+\deg(b_j(x;x_1,\dots,x_j))-1.\]
\end{proof}

\begin{corollary}
The maps 
\begin{equation*}
M_j':\s\OO^{\otimes j}\to \s\OO,\, (x_1,\dots, x_j)\mapsto b_j(m;x_1,\dots, x_j)
\end{equation*}
for $j>1$ and the map
\begin{equation*}
M_1':\s\OO\to \s\OO,\, x\mapsto b_1(m;x)-(-1)^{|x|}b_1(m;x)
\end{equation*}
are homogeneous of degree 1. 
\end{corollary}
\begin{proof}
For $j>1$ it is a direct consequence of \Cref{lemmadegree}. For $j=1$ this is a computation very similar to that required for \Cref{lemmadegree}. 
\end{proof}

The problem we have encountered with the degrees can be resolved using shift maps as the following proposition shows. Recall that we have shift maps $A\to SA$ of degree 1 given by the identity. 

\begin{propo}\label{ainftystructure}
If $\OO$ is an operad with an $A_\infty$-multiplication $m\in\OO$, then there is an $A_\infty$-algebra structure on the shifted module $S\s\OO$. 
\end{propo}
\begin{proof}
As in the proof of \Cref{lemmadegree}, a way to turn $M'_j$ into a map of degree $2-j$ is introducing a grading on $\s\OO$ given by arity plus internal degree (without subtracting 1). This is equivalent to defining an $A_\infty$-algebra structure $M$ on $S\s\OO$ shifting the map $M'=M'_1+M'_2+\cdots$, where $S$ is the shift of graded modules. Therefore, we define $M_j$ to be the map making the following diagram commute.

\[
\begin{tikzcd}
(S\s\OO)^{\otimes j}\arrow[r,"M_j"]\arrow[d, "(S^{\otimes j})^{-1}"'] & S\s\OO\\
\s\OO^{\otimes j}\arrow[r, "M'_j"] & \s\OO\arrow[u,"S"']
\end{tikzcd}
\]

In other words, $M_j=\overline{\sigma}(M'_j)$, where $\overline{\sigma}(F)=S\circ F\circ (S^{\otimes n})^{-1}$ for $F\in\End_{\s\OO}(n)$ is the map inducing an isomorphism $\End_{\s\OO}\cong \s\End_{S\s\OO}$, see \Cref{thm:shift} and \Cref{eq:barsigma}. Since $\overline{\sigma}$ is an operad morphism, for $M=M_1+M_2+\cdots$, we have

\[
M\tilde{\circ}M=\overline{\sigma}(M')\tilde{\circ}\overline{\sigma}(M')=\overline{\sigma}(M'\circ M')=0.
\]

So now we have that $M\in\s\End_{S\s\OO}$ is an element of natural degree 1 concentrated in positive arity such that $M\tilde\circ M=0$. Therefore, as by definition an $A_\infty$-structure on $A$ is the same as a map of operads $\mathcal{A}_\infty \longrightarrow \End_A$, $M$ is the desired $A_\infty$-algebra structure on $S\s\OO$. 
\end{proof}

\begin{remark}
Note that $M$ is defined as an structure map on $S\s\OO$. This kind of shifted operad is called \emph{odd operad} in \cite{ward}. This means that $S\s\OO$ is not an operad anymore, since the associativity relation for graded operads involves signs that depend on the degrees, which are now shifted. 
\end{remark}

We have defined $A_\infty$-structure maps $M_j\in\s\End_{S\s\OO}$. Now we can use the brace structure of the operad $\s\End_{S\s\OO}$ to get a $A_\infty$-algebra structure given by maps
\begin{equation}\label{barmaps}
\overline{M}_j:(S\s\End_{S\s\OO})^{\otimes j}\to S\s\End_{S\s\OO}
\end{equation}
by applying $\overline{\sigma}$ to maps
\[\overline{M}'_j:(\s\End_{S\s\OO})^{\otimes j}\to \s\End_{S\s\OO}\]
defined as
\[
\overline{M}'_j(f_1,\dots,f_j)=\overline{B}_j(M;f_1,\dots, f_j),\,\, j>1,\,\,\,\,\,\,\,
\overline{M}'_1(f)=\overline{B}_1(M;f)-(-1)^{|f|}\overline{B}_1(f;M),
\]
where $\overline{B}_j$ denotes the brace map on $\s\End_{S\s\OO}$.

We now define the Hochschild complex as done by Ward in \cite{ward}.
\begin{defin}
The \emph{Hochschild cochains} of a graded module $A$ are defined to be the graded module $S\s\End_A$. If $(A,d)$ is a cochain complex, then $S\s\End_A$ is endowed with a differential \[\partial(f)=[d,f]=d\circ f -(-1)^{|f|}f\circ d\] where $|f|$ is the natural degree of $|f|$ and $\circ$ is the plethysm operation given by insertions.
\end{defin}
In particular, $S\s\End_{S\s\OO}$ is the module of Hochschild cochains of $S\s\OO$. If $\OO$ has an $A_\infty$-multiplication, then the differential of the Hochschild complex is $\overline{M}_1$ from \Cref{barmaps}.
\begin{remark}
The functor $S\s$ is called the ``oddification'' of an operad in the literature \cite{wardthesis}. 
The reader might find odd to define the Hochschild complex in this way instead of just $\End_A$. The reason is that operadic suspension provides the necessary signs and the extra shift gives us the appropriate degrees. In addition, this definition allows the extra structure to arise naturally instead of having to define the signs by hand. For instance, if we have an associative multiplication $m_2\in\End_A(2)=\Hom(A^{\otimes 2},A)$, the element $m_2$ would not satisfy the equation $m_2\circ m_2=0$ and thus cannot be used to induce a multiplication on $\End_A$ as we did above.
\end{remark}

 A natural question to ask is what relation there is between the $A_\infty$-algebra structure on $S\s\OO$ and the one on $S\s\End_{S\s\OO}$. In \cite{GV} it is claimed that given an operad $\OO$ with an $A_\infty$-multiplication, the map

\[
\OO \to \End_\OO,\, x\mapsto \sum_{n\geq 0}b_n(x;-)
\]
is a morphism of $A_\infty$-algebras. In the associative case, this result leads to the definition of homotopy $G$-algebras, which connects with the classical Deligne conjecture. We are going to adapt the statement of this claim to our context and prove it. This way we will obtain an $A_\infty$-version of homotopy $G$-algebras and consequently an $A_\infty$-version of the Deligne conjecture. Let $\Phi'$ the map defined as above but on $\s\OO$, i.e.
\[
\Phi'\colon\s\OO \to \End_{\s\OO},\, x\mapsto \sum_{n\geq 0}b_n(x;-).
\]
Let $\Phi:S\s\OO\to S\s\End_{S\s\OO}$ the map making the following diagram commute
\begin{equation}\label{Phi}
\begin{tikzcd}
S\s\OO\arrow[rr, "\Phi"]\arrow[d] & & S\s\End_{S\s\OO}\\
\s\OO\arrow[r, "\Phi'"]& \End_{\s\OO}\arrow[r, "\cong"]& \s\End_{S\s\OO}\arrow[u]
\end{tikzcd}
\end{equation}
where the isomorphism $\End_{\s\OO}\cong\s\End_{S\s\OO}$ is given in \Cref{thm:shift}. 

\begin{remark}
We have only used the operadic structure on $\s\OO$ to define an $A_\infty$-algebra structure on $S\s\OO$, so the constructions and results in these sections are valid if we replace $\s\OO$ by any graded module $A$ such that $SA$ is an $A_\infty$-algebra. 
\end{remark}

\begin{thm}\label{theorem}
The map $\Phi$ defined in diagram (\ref{Phi}) above is a morphism of $A_\infty$-algebras, i.e. for all $j\geq 1$ the equation

\[\Phi(M_j)=\overline{M}_j(\Phi^{\otimes j})\]
holds, where the $M_j$ is the $j$-th component of the $A_\infty$-algebra structure on $S\s\OO$ and $\overline{M}_j$ is the $j$-th componnent of the $A_\infty$-algebra structure on $S\s\End_{S\s\OO}$. 
\end{thm}
\begin{proof}
Let us have a look at the following diagram

\begin{equation}\label{proofdiagram}
\begin{tikzcd}[row sep = 1.5em]
(S\s\OO)^{\otimes j}\arrow[dr,red] \arrow[ddd, "M_j"']\arrow[rrrr, "\Phi^{\otimes j}"]& & & & (S\s\End_{S\s\OO})^{\otimes j}\arrow[ddd,  "\overline{M}_j"]\\
&\s\OO^{\otimes j}\arrow[r,blue, "(\Phi')^{\otimes j}"]\arrow[d, blue, "M'_j"] & (\End_{\s\OO})^{\otimes j}\arrow[r, blue,"\overline{\sigma}^{\otimes j}"] \arrow[d, dashed, "\mathcal{M}_j",blue]& (\s\End_{S\s\OO})^{\otimes j}\arrow[ur,red]\arrow[d, "\overline{M}'_j",blue]& \\
&\s\OO\arrow[r, blue, "\Phi'"]& \End_{\s\OO} \arrow[r, blue, "\overline{\sigma}"] & \s\End_{S\s\OO}\arrow[dr,red]& \\
S\s\OO\arrow[rrrr,  "\Phi"']\arrow[ur,red]& & & & S\s\End_{S\s\OO}
\end{tikzcd}
\end{equation}

where the diagonal red arrows are shifts of graded $R$-modules. We need to show that the diagram defined by the external black arrows commutes. But these arrows are defined so that they commute with the red and blue arrows, so it is enough to show that the inner blue diagram commutes. The blue diagram can be split into two different squares using the dashed arrow $\mathcal{M}_j$ that we are going to define next, so it will be enough to show that the two squares commute. 

 The map 
$\mathcal{M}_j:(\End_{\s\OO})^{\otimes j}\to\End_{\s\OO}$
is defined by 
\begin{align*}
&\mathcal{M}_j(f_1, \dots, f_j)=B_j(M';f_1,\dots, f_j) &\text{ for }j>1,\\
&\mathcal{M}_1(f)=B_1(M';f)-(-1)^{|f|}B_1(f;M'),
\end{align*}
 where $B_j$ is the natural brace structure map on the operad $\End_{\s\OO}$, i.e. for $f\in\End_{\s\OO}(n)$, 
\[B_j(f;f_1,\dots, f_j)=\sum_{k_0+\cdots+k_j=n-j} f(1^{\otimes k_0}\otimes f_1\otimes 1^{\otimes k_1}\otimes\cdots\otimes f_j\otimes 1^{\otimes k_j}).\]
 The $1$'s in the braces are identity maps. In the above definition, $|f|$ denotes the degree of $f$ as an element of $\End_{\s\OO}$, which is the same as the degree $\overline{\sigma}(f)\in \s\End_{S\s\OO}$ because $\overline{\sigma}$ is an isomorphism, as mentioned in \Cref{barsigma}.  
 
 The inner square of diagram (\ref{proofdiagram}) is divided into two halves, so we divide the proof into two as well, showing the commutativity of each half independently.

 \subsection*{\centering{Commutativity of the right blue square} }
 Let us show now that the right square commutes. Recall that $\overline{\sigma}$ is an isomorphism of operads and $M=\overline{\sigma}(M')$. Then we have for $j>1$
 \[\overline{M}'_j(\overline{\sigma}(f_1),\dots,\overline{\sigma}(f_j))=\overline{B}_j(M;\overline{\sigma}(f_1),\dots,\overline{\sigma}(f_j))=\overline{B}_j(\overline{\sigma}(M');\overline{\sigma}(f_1),\dots,\overline{\sigma}(f_j)).\]
 
 Now, since the brace structure is defined as an operadic composition, it commutes with $\overline{\sigma}$, so
 \[\overline{B}_j(\overline{\sigma}(M');\overline{\sigma}(f_1),\dots,\overline{\sigma}(f_j))=\overline{\sigma}(B_j(M';f_1,\dots, f_j))=\overline{\sigma}(\mathcal{M}_j(f_1,\dots, f_j)),\]
 and therefore the right blue square commutes for $j>1$. For $j=1$ the result follows analogously.

The proof that the left blue square commutes consists of several lengthy calculations so we are going to devote the next section to that. However, it is worth noting that the commutativity of the left square does not depend on the particular operad $\s\OO$, so it is still valid if $m$ satisfies $m\circ m=0$ for any circle operation defined in terms of insertions. This is essentially the original statement in \cite{GV}.
\subsection*{\centering{Commutativity of the left blue square}}
We are going to show here that the left blue square in diagram (\ref{proofdiagram}) commutes, i.e. that 

\begin{equation}\label{commutative}
\Phi'(M'_j)=\mathcal{M}_j((\Phi')^{\otimes j})
\end{equation}

for all $j\geq 1$. First we prove the case $j>1$. Let $x_1,\dots, x_j\in \s\OO^{\otimes j}$. We have on the one hand

\begin{align*}
\Phi'(M'_j(x_1,\dots, x_j))& = \Phi'(b_j(m;x_1,\dots, x_j))=\sum_{n\geq 0} b_n(b_j(m;x_1,\dots, x_j);-)\\
&=\sum_n\sum_l\sum b_l(m; -, b_{i_1}(x_1;-),\cdots,b_{i_j}(x_j;-),-)
\end{align*}
where $l=n-(i_1+\cdots+i_j)+j$. The sum with no subindex runs over all the possible order-preserving insertions. Note that $l\geq j$. Evaluating the above map on elements would yield Koszul signs coming from the brace relation. Also recall from \Cref{lemmadegree} that $|b_j(x;-)|=|x|$. Now, fix some value of $l\geq j$ and let us compute the $M'_l$ component of

\begin{align*}
\mathcal{M}_j(\Phi'(x_1),\dots, \Phi'(x_j))=B_j(M';\Phi'(x_1),\dots, \Phi'(x_j))
\end{align*}

that is, $B_j(M'_l;\Phi'(x_1),\dots, \Phi'(x_j))$. By definition, this equals

\begin{align*}
\sum M'_l(-,\Phi'(x_1),\cdots, \Phi'(x_j),-)=&\sum_{i_1,\dots, i_j}\sum M'_l(-,b_{i_1}(x_1;-),\cdots,b_{i_j}(x_j;-),-)\\
=&\sum_{i_1,\dots, i_j}\sum b_l(m;-,b_{i_1}(x_1;-),\cdots,b_{i_j}(x_j;-),-).
\end{align*}

We are using hyphens instead of $1$'s to make the equality of both sides of the \Cref{commutative} more apparent, and to make clear that when evaluating on elements those are the places where the elements go. 

For each tuple $(i_1,\dots, i_j)$ we can choose $n$ such that $n-(i_1+\cdots+i_j)+j=l$, so the above sum equals

\[\underset{\mathclap{n-(i_1+\cdots+i_j)+j=l}}{\sum_{n,i_1,\dots, i_j}}\sum b_l(m;-,b_{i_1}(x_1;-),\cdots,b_{i_j}(x_j;-),-).\]

So each $M'_l$ component for $l\geq j$ produces precisely the terms $b_l(m;\dots)$ appearing in $\Phi'(M'_j)$. Conversely, for every $n\geq 0$ there exists some tuple $(i_1,\dots, i_j)$ and some $l\geq j$ such that $n$ is the that $n-(i_1+\cdots+i_j)+j=l$, so we do get all the summands from the left hand side of \Cref{commutative}, and thus we have the equality $\Phi'(M'_j)=\mathcal{M}_j((\Phi')^{\otimes j})$ for all $j>1$.

It is worth treating the case $n=0$ separately since in that case we have the summand $b_0(b_j(m;x_1,\dots, x_j))$
in $\Phi'(b_j(m;x_1,\dots, x_j))$, where we cannot apply the brace relation. This summand is equal to \[B_j(M'_j;b_0(x_1),\dots, b_0(x_j))=M'_j(b_0(x_1),\dots, b_0(x_j))=b_j(m;b_0(x_1),\dots, b_0(x_j)),\] since by definition $b_0(x)=x$.

Now we are going to show the case $j=1$, that is

\begin{equation}\label{case1}
\Phi'(M'_1(x))=\mathcal{M}_1(\Phi'(x)).
\end{equation} 

This is going to be divided into two parts, since $M'_1$ has two clearly distinct summands, one of them consisting of braces of the form $b_l(m;\cdots)$ (insertions in $m$) and another one consisting of braces of the form $b_l(x;\cdots)$ (insertions in $x$). We will therefore show that both types of braces cancel on each side of \Cref{case1}.

\subsubsection*{Insertions in $m$}

Let us first focus on the insertions in $m$ that appear in \Cref{case1}. Recall that 

\begin{equation}\label{phim}
\Phi'(M'_1(x))=\Phi'([m,x])=\Phi'(b_1(m;x))-(-1)^{|x|}\Phi'(b_1(x;m))
\end{equation}

so we focus on the first summand

\begin{align*}
\Phi'(b_1(m;x))=&\sum_n b_n(b_1(m;x);-)=\sum_n \underset{n\geq i}{\sum_i} \sum b_{n-i+1}(m;-, b_i(x;-),-)\\
=&\smashoperator[r]{\sum_{\substack{n,i\\n-i+1> 0}}}\sum b_{n-i+1}(m;-, b_i(x;-),-)
\end{align*}

where the sum with no indices runs over all the positions in which $b_i(x;-)$ can be inserted (from $1$ to $n-i+1$ in this case).

On the other hand, since $|\Phi'(x)|=|x|$, the right hand side of \Cref{case1} becomes

\begin{equation}\label{mphi}
\mathcal{M}_1(\Phi'(x))=B_1(M';\Phi'(x))-(-1)^{|x|}B_1(\Phi'(x);M').
\end{equation}

Again, we are focusing now on the first summand, but with the exception of the part of $M_1'$ that corresponds to $b_1(\Phi(x);m)$. From here the argument is a particular case of the proof for $j>1$, so the terms of the form $b_l(m;\cdots)$ are the same on both sides of the \Cref{case1}. 

\subsubsection*{Insertions in $x$}

And now, let us study the insertions in $x$ that appear in \Cref{case1}. We will check that insertions in $x$ from the left hand side and right hand side cancel. Let us look first at the left hand side. From $\Phi'(M'_1(x))$ in \Cref{phim} we had 

\[-(-1)^{|x|}\Phi'(b_1(x;m))=-(-1)^{|x|}\sum_n b_n(b_1(x;m);-).\]

The factor $-(-1)^{|x|}$ is going to appear everywhere, so we may cancel it. Thus we just have

\[\Phi'(b_1(x;m))=\sum_n b_n(b_1(x;m);-).\]
We are going to evaluate each term of the sum, so let $z_1,\dots, z_n\in \s\OO$. We have by the brace relation that

\begin{align}\label{insertionx1}
b_n(b_1(x;m);z_1,\dots, z_n)=\sum_{l+j=n+1}\smashoperator[l]{\sum_{i=1}^{n-j+1}}&(-1)^{\varepsilon} b_l(x;z_1,\dots,b_j(m;z_{i},\dots, z_{i+j}),\dots, z_n)\nonumber\\
 &+\sum_{i=1}^{n+1}(-1)^{\varepsilon}b_{n+1}(x;z_1,\dots, z_{i-1},m,z_i,\dots, z_n),
\end{align}

where $\varepsilon$ is the usual Koszul sign with respect to the grading in $\s\OO$. We have to check that the insertions in $x$ that appear in $\mathcal{M}_1(\Phi'(x))$ (right hand side of the \cref{case1}) are exactly those in \Cref{insertionx1} above (left hand side of \cref{case1}).

Therefore let us look at the right hand side of \Cref{case1}. Here we will study the cancellations from each of the two summands that naturally appear. From \Cref{mphi}, i.e. $\mathcal{M}_1(\Phi'(x))=B_1(M';\Phi'(x))-(-1)^{|x|}B_1(\Phi'(x);M')$  we have 
\[-(-1)^{|x|}b_1(\Phi'(x);m)=-(-1)^{|x|}\sum_n b_1(b_n(x;-);m)\] 
coming from the first summand since $B_1(M'_1;\Phi'(x))=M'_1(\Phi'(x))$. We are now only interested in insertions in $x$. Again, cancelling $-(-1)^{|x|}$ we get
\[b_1(\Phi'(x);m)=\sum_n b_1(b_n(x;-);m).\] 
Each term of the sum can be evaluated on $(z_1,\dots, z_n)$ to produce

\begin{align}\label{insertionx2}
b_1(b_n(x;z_1, \dots, z_n);m)&=\\
\sum_{i=1}^n (-1)^{\varepsilon+|z_i|}b_n(x;z_1,\dots, b_1(z_i;m),\dots, z_n)&+\sum_{i=1}^{n+1} (-1)^{\varepsilon}b_{n+1}(x;z_1,\dots, z_{i-1},m,z_{i},\dots, z_n)\nonumber
\end{align}

Note that we have to apply the Koszul sign rule twice: once at evaluation, and once more to apply the brace relation. Now, from the second summand of $\mathcal{M}_1(\Phi'(x))$ in the right hand side of \cref{mphi}, after cancelling $-(-1)^{|x|}$ we obtain

\begin{align*}
B_1(\Phi'(x);M')=&\sum_l B_1(b_l(x;-);M')=\sum_l\sum b_l(x;-,M',-) \\
=&\left(\sum_{j> 1} \sum_l\sum b_l(x;-,b_j(m;-),-)+\sum_l\sum b_l(x;-,b_1(-;m),-)\right).
\end{align*}
We are going to evaluate on $(z_1,\dots, z_n)$ to make this map more explicit, giving us
 
 \begin{align}\label{lhs2}
 \sum_{l+j=n+1}\sum_{i=1}^{n-j+1}(-1)^{\varepsilon} b_l(x;z_1,\dots,b_j(m;z_{i},\dots, z_{i+j}),\dots, z_n) -\sum_{i=1}^{n} (-1)^{\varepsilon+|z_i|}b_n(x;z_1,\dots,b_1(z_{i};m),\dots, z_n).
 \end{align}
 
 The minus sign comes from the fact that $b_1(z_i;m)$ comes from $M'_1(z_i)$, so we apply the signs in the definition of $M'_1(z_i)$. We therefore have that the right hand side of \cref{mphi} is the result of adding equations (\ref{insertionx2}) and (\ref{lhs2}). After this addition we can see that the first sum of \cref{insertionx2} cancels the second sum of \cref{lhs2}. 

 We also have that the second sum in \cref{insertionx2} is the same as the second sum in \cref{insertionx1}, so we are left with only the first sum of \cref{lhs2}. This is the same as the first sum in \cref{insertionx1}, so we have already checked that the equation $\Phi'(M'_1)=\mathcal{M}_1(\Phi')$ holds. 
  
 In the case $n=0$, we have to note that $B_1(b_0(x);m)$ vanishes because of arity reasons: $b_0(x)$ is a map of arity 0, so we cannot insert any inputs. This finishes the proof.
 \end{proof}

We have given an implicit definition of the components of the $A_\infty$-algebra structure on $S\s\OO$, namely, \[M_j=\overline{\sigma}(M'_j)=(-1)^{\binom{j}{2}}S\circ M'_j\circ(S^{-1})^{\otimes j},\]
but it is useful to have an explicit expression that determines how it is evaluated on elements of $S\s\OO$. We will need these explicit expressions to describe $J$-algebras, which are $A_\infty$-version of homotopy $G$-algebras. This way we can state the $A_\infty$-Deligne conjecture in a more precise way. These explicit formulas will also clear up the connection with the work of Gerstenhaber and Voronov. We hope that these explicit expressions can be useful to perform calculations in other mathematical contexts where $A_\infty$-algebras are used.

The following can be shown using a straightforward sign calculation.
\begin{lem}\label{explicit}
For $x,x_1,\dots,x_n\in\s\OO$, we have the following expressions.

\begin{align*}
&M_n(Sx_1,\dots, Sx_n)=(-1)^{\sum_{i=1}^n(n-i)|x_i|}Sb_n(m;x_1,\dots, x_n) & & n>1,\\
&M_1(Sx)=Sb_1(m;x)-(-1)^{|x|}Sb_1(x;m).
\end{align*}

Here $|x|$ is the degree of $x$ as an element of $\s\OO$, i.e. the natural degree.  \qed
\end{lem}

Now that we have the explicit formulas for the $A_\infty$-structure on $S\s\OO$ we can state and prove an $A_\infty$-version of the Deligne conjecture. Let us first re-adapt the definition of homotopy $G$-algebra from \cite[Definition 2]{GV} to our conventions.

\begin{defin}\label{homotopygalgebras}
A \emph{homotopy $G$-algebra} is differential graded algebra $V$ with a differential $M_1$ and a product $M_2$ such that the shift $S^{-1}V$ is a brace algebra with brace maps $b_n$. The product differential and the product must satisfy the following compatibility identities. Let $x,x_1,x_2,y_1,\dots, y_n\in S^{-1}V$. We demand 
\begin{align*}
Sb_n(&S^{-1}M_2(Sx_1,Sx_2);y_1,\dots, y_n) = 
&\sum_{k=0}^n (-1)^{(|x_2|+1)\sum_{i=1}^k|y_i|}M_2(b_k(x_1;y_1,\dots, y_k),b_{n-k}(x_2;y_{k+1},\dots, y_n))
\end{align*}
and
\begin{align*}
S&b_n(S^{-1}M_1(Sx);y_1,\dots, y_n)-M_1(Sb_n(x;y_1,\dots,y_n))
-(-1)^{|x|+1}\sum_{p=1}^n(-1)^{\sum_{i=1}^p|y_i|}Sb_n(x;y_1,\dots,M_1(Sy_p),\dots, y_n)\\
=&-(-1)^{(|x|+1)|y_1|}M_2(Sy_1,Sb_{n-1}(x;y_2,\dots, y_n))\\
 &+(-1)^{|x|+1}\sum_{p=1}^{n-1}(-1)^{n-1+\sum_{i=1}^p|y_i|}Sb_{n-1}(x;y_1,\dots,M_2(Sy_p,Sy_{p+1}),\dots y_n)\\
 &-(-1)^{|x|+\sum_{i=1}^{n-1}|y_i|}M_2(Sb_{n-1}(x;y_1,\dots, y_{n-1}),Sy_n).
\end{align*}
\end{defin}

Notice that our signs are slightly different to those in \cite{GV} as a consequence of our conventions. Our signs will be a particular case of those in \Cref{Jalgebras}, which are set so that \Cref{ainftydeligne} holds in consistent way with operadic suspension and all the shifts that the authors of \cite{GV} do not consider.

We now introduce $J$-algebras as an $A_\infty$-generalization of homotopy $G$-algebras. This will allow us to generalize the Deligne conjecture to the $A_\infty$-setting. 

\begin{defin}\label{Jalgebras}
A \emph{$J$-algebra} $V$ is an $A_\infty$-algebra with structure maps $\{M_j\}_{j\geq 1}$ such that the shift $S^{-1}V$ is a brace algebra. Furthermore, the braces and the $A_\infty$-structure satisfy the following compatibility relations. Let $x, x_1,\dots, x_j, y_1,\dots, y_n\in S^{-1}V$. For $n\geq 0$ we demand 

\begin{align*}
(-1)^{\sum_{i=1}^n(n-i)|y_i|}Sb_n(S^{-1}&M_1(Sx);y_1,\dots, y_n)=\\
&\underset{\mathclap{1\leq i_1\leq n-k+1}}{\sum_{\mathclap{l+k-1=n}}}(-1)^{\varepsilon}M_l(Sy_1,\dots, Sb_{k}(x;y_{i_1},\dots),\dots, Sy_n)\\
-(-1)^{|x|}\underset{\mathclap{1\leq i_1\leq n-k+1}}{\sum_{\mathclap{l+k-1=n}}}&(-1)^{\eta} Sb_k(x;y_1,\dots, S^{-1}M_l(Sy_{i_1},\dots,), \dots, y_n),
\end{align*}
where
\begin{align*}
\varepsilon = &\sum_{v=1}^{i_1-1}|y_v|(|x|-k+1)+\sum_{v=1}^{k}|y_{i_1+v-1}|(k-v)+(l-i_1)|x|.
\end{align*}
and
\begin{align*}
\eta=& \sum_{v=1}^{i_1-1}(k-v)|y_v|+\sum_{v=1}^{i_1-1}l|y_v|+\sum_{v=i_1}^{i_1+l-1}(k-i_1)|y_v|+\sum_{v=i_1}^{n-l}(k-v)|y_{v+l}|.
\end{align*}
For $j>1$ we demand
\begin{align*}
&(-1)^{\sum_{i=1}^n(n-i)|y_i|}Sb_n(S^{-1}M_j(Sx_1,\dots, Sx_j);y_1,\dots, y_n)=\\
&\sum(-1)^{\varepsilon}M_l(Sy_1,\dots, Sb_{k_1}(x_1;y_{i_1},\dots),\dots, Sb_{k_j}(x_j;y_{i_j},\dots),\dots, Sy_n).
\end{align*}
The unindexed sum runs over all possible choices of non-negative integers that satisfy $l+k_1+\cdots+k_j-j=n$ and over all possible ordering-preserving insertions. The right hand side sign is given by
\begin{align*}
\varepsilon =& \underset{1\leq v\leq k_t}{\sum_{\mathclap{1\leq t\leq j}}} |y_{i_t+v-1}|(k_t-v)+\sum_{\mathclap{1\leq v< l\leq j}}k_v|x_l|+\sum_{\mathclap{1\leq v\leq l\leq j}} |x_v|(i_{l+1}-i_l-k_l)\\
&+\underset{\mathclap{i_t\leq v< i_{t+1}}}{\sum_{\mathclap{0\leq t< l\leq j}}}(|y_v|+1)(|x_l|-k_l+1)+\sum_{\mathclap{0\leq v<l\leq j}}(i_{v+1}-i_v-k_v)(|x_l|-k_l+1)\
\end{align*}
In the sums we are setting $i_0=0$ and $i_{j+1}=n+1$.
\end{defin}

With this in place, we can now show one of our main results.

\begin{corollary}[The $A_\infty$-Deligne conjecture]\label{ainftydeligne}
If $A$ is an $A_\infty$-algebra, then its Hochschild complex $S\s\End_A$ is a $J$-algebra.
\end{corollary}
\begin{proof}
We know that $\s\End_A$ is a brace algebra as it is an operad. Since $A$ is an $A_\infty$-algebra, the structure map $m=m_1+m_2+\cdots$ determines an $A_\infty$-multiplication $m\in\s\End_A$. It follows by \Cref{ainftystructure} that $S\s\End_A$ is an $A_\infty$-algebra. Therefore, we need to show the compatibility relations. The result follows by direct computation from \Cref{theorem}, expanding the definitions and taking into account the Koszul sign rule. 
\end{proof}

\section{The derived $A_\infty$-structure on an operad}\label{sec:derivedstructure}

 In this section we finally establish the connection between classical and derived $A_\infty$-algebras. In \Cref{derivedmaps} we are able to obtain explicit derived $A_\infty$-maps on $A=S\s\OO$ for a sufficiently bounded operad $\OO$ with a derived $A_\infty$-multiplication. This opens the door to the formulation and proof on a new version of the Deligne conjecture in \Cref{dainftydeligne}.

We begin by stating one of the key ingredients.

\begin{propo}\label{whitehouse}
Let $(A, d^A) \in\tc^b$ be a twisted complex horizontally bounded on the right and $A$ its underlying
cochain complex. We have natural bijections 
\begin{align*}
\Hom_{\mathrm{vbOp},d^A}(d\calA_\infty,\End_A) &\cong
\Hom_{\mathrm{vbOp}}(\calA_\infty, \uEnd_A)\\
&\cong \Hom_{\mathrm{vbOp}}(\calA_\infty, \uEnd_{\Tot(A)})\\
&\cong \Hom_{\mathrm{fCOp}}(\calA_\infty,\underline{\End}_{\Tot(A)}),
\end{align*}
where $\vbOp$ and $\fCOp$ denote the categories of operads in $\vbc$ and $\fc$ respectively, and $\Hom_{\vbOp,d^A}$
denotes the subset of morphisms which send $\mu_{i1}$ to $d^A_i$. We view $\mathcal{A}_\infty$ as an operad in $\vbc$ sitting in
horizontal degree zero or as an operad in filtered complexes with trivial filtration.
\end{propo}

\begin{proof}
This follows from the proof of  \cite[Proposition 4.55]{whitehouse} adapted to our case, which shows that there is a derived $A_\infty$-structure on $A=S\s\OO$, see \Cref{derivedmultiplication}. We refer the reader to \Cref{sec:bigraded}, \Cref{sec:total} and \Cref{sec:enrichment} to recall the definitions of the categories used. 
\end{proof}

\begin{remark}\label{boundednessremark}
According to \Cref{filterversion}, the last isomorphism can be replaced by 
\[\Hom_{\mathrm{vbOp}}(\calA_\infty, \uEnd_{\Tot(A)})\cong \Hom_{\mathrm{COp}}(\calA_\infty,\End_{\Tot(A)}),\]
where $\mathrm{COp}$ is the category of operads in cochain complexes. 
\end{remark}
There are several important assumptions to make in order to use \Cref{whitehouse}. First of all, we need $A$ to be horizontally bounded on the right, meaning that there exists some integer $i$ such that $A_k^{d-k}=0$ for all $k>i$. In our case, $A=S\s\OO$ for $\OO$ an operad with a derived $A_\infty$-multiplication, so being horizontally bounded on the right implies that for each $j>0$ we can only have finitely many non-zero components $m_{ij}$. This situation happens in practice in all known examples of derived $A_\infty$-algebras so far, some of them are in \cite[Remark 6.5]{muro}, \cite{RW}, and \cite[\S 5]{women}. Under this assumption we may replace all direct products by direct sums.

We also need to provide $A$ with a twisted complex structure. The reason for this is that \Cref{whitehouse} uses the definition of derived $A_\infty$-algebras on an underlying twisted complex, see \Cref{equivalent}. We show explicitly the existence of a twisted complex structure on an operad with derived $A_\infty$-multiplication in \Cref{twistedoperad}, but it also follows from \Cref{mi1}. We also provide another version of this theorem that works for bigraded modules, \Cref{alternative}. 

With these assumption, by \Cref{whitehouse} we can guarantee the existence of a derived $A_\infty$-algebra structure on $A$ provided that $\Tot(A)$ has an $A_\infty$-algebra structure. Note that we abuse of notation and identify $x_1\otimes\cdots\otimes x_j$ with an element of $\Tot(A^{\otimes j})$ with only one non-zero component. For a general element, extend linearly.

\begin{thm}\label{derivedmaps}
Let $A=S\s\OO$ where $\OO$ is an operad horizontally bounded on the right with a derived $A_\infty$-multiplication $m=\sum_{ij}m_{ij}\in\OO$. Let $x_1\otimes\cdots\otimes x_j\in (A^{\otimes j})^{d-k}_k$ and let $x_v = Sy_v$ for $v=1,\dots, j$ and $y_v$ be of bidegree $(k_v,d_v-k_v)$. The following maps $M_{ij}$ for $j\geq 2$ determine a derived $A_\infty$-algebra structure on $A$.
\[M_{ij}(x_1,\dots,x_j)= (-1)^{\sum_{v=1}^j(j-v)(d_v-k_v)}\sum_lSb_j(m_{il};y_1,\dots, y_j). \]
\end{thm}
\begin{proof}
Since $m$ is a derived $A_\infty$-multiplication $\OO$, we have that $m\star m=0$ when we view $m$ as an element of $\Tot(\s\OO)$. By \Cref{ainftystructure}, this defines an $A_\infty$-algebra structure on $S\Tot(\s\OO)$ given by maps
\[M_j:(S\Tot(\s\OO))^{\otimes j}\to S\Tot(\s\OO)\]
induced by shifting brace maps
\[b_j^\star(m;-):(\Tot(\s\OO))^{\otimes j}\to \Tot(\s\OO).\]
 The graded module $S\Tot(\s\OO)$ is endowed with the structure of a filtered complex with differential $M_1$ and filtration induced by the column filtration on $\Tot(\s\OO)$. Note that $b^\star_j(m;-)$ preserves the column filtration since every component $b^\star_j(m_{ij};-)$ has positive horizontal degree. 
 
Since $S\Tot(\s\OO)\cong \Tot(S\s\OO)$, we can view $M_j$ as the image of a morphism of operads of filtered complexes $f:\mathcal{A}_\infty\to \End_{\Tot(S\s\OO)}$ in such a way that $M_j=f(\mu_j)$ for $\mu_j\in\mathcal{A}_\infty(j)$. 

We now work our way backwards using the strategy also employed by the proof of \Cref{whitehouse}. The isomorphism 
\[\Hom_{\mathrm{vbOp}}(\calA_\infty, \uEnd_{\Tot(A)})\cong \Hom_{\mathrm{COp}}(\calA_\infty,\End_{\Tot(A)})\]
does not modify the map $M_j$ at all but allows us to see it as a element of $\uEnd_{\Tot(A)}$ of bidegree $(0,2-j)$. 

The isomorphism 
\[\Hom_{\mathrm{vbOp}}(\calA_\infty, \uEnd_A)\cong \Hom_{\mathrm{vbOp}}(\calA_\infty, \uEnd_{\Tot(A)})\] 
in the direction we are following is the result of applying $\Hom_{\vbOp}(\calA_\infty,-)$ to the map described in \Cref{composition}. Under this isomorphism, $f$ is sent to the map \[\mu_j\mapsto \mathfrak{Tot}^{-1}\circ c(M_j,\mu^{-1})=\mathfrak{Tot}^{-1}\circ M_j\circ \mu^{-1},\] where $c$ is the composition in $\ufC$. The functor $\mathfrak{Tot}^{-1}$ decomposes $M_j$ into a sum of maps $M_j=\sum_i \widetilde{M}_{ij}$, where each $\widetilde{M}_{ij}$ is of bidegree $(i,2-j-i)$. More explicitly, let $A=S\s\OO$ and let $x_1\otimes\cdots\otimes x_j\in (A^{\otimes j})^{d-k}_k$. 

Then we have
\begin{align}\label{totsign}
\mathfrak{Tot}^{-1}(M_j( \mu^{-1}(x_1\otimes\cdots\otimes x_j)))=
\mathfrak{Tot}^{-1}(Sb_j^\star(m;(S^{-1})^{\otimes j}(\mu^{-1}(x_1\otimes\cdots\otimes x_j))))&\nonumber\\
=\sum_i(-1)^{id}\sum_l Sb_j^\star(m_{il};(S^{-1})^{\otimes j}(\mu^{-1}(x_1\otimes\cdots\otimes x_j)))&\nonumber\\
=\sum_i(-1)^{id}\sum_l(-1)^{\varepsilon} Sb_j(m_{il};(S^{-1})^{\otimes j}(\mu^{-1}(x_1\otimes\cdots\otimes x_j)))&\nonumber\\
=\sum_i\sum_l(-1)^{id+\varepsilon} Sb_j(m_{il};(S^{-1})^{\otimes j}(\mu^{-1}(x_1\otimes\cdots\otimes x_j)))
\end{align}
so that \[\widetilde{M}_{ij}(x_1,\dots,x_j)=\sum_l(-1)^{id+\varepsilon} Sb_j(m_{il};(S^{-1})^{\otimes j}(\mu^{-1}(x_1\otimes\cdots\otimes x_j))),\] where $b_j$ is the brace on $\s\OO$ and $\varepsilon$ is given in \Cref{totcomp}. 
According to the isomorphism 
\begin{equation}\label{firstiso}
\Hom_{\mathrm{vbOp},d^A}(d\calA_\infty,\End_A)\cong
\Hom_{\mathrm{vbOp}}(\calA_\infty, \uEnd_A),
\end{equation}
 the maps $M_{ij}=(-1)^{ij}\widetilde{M}_{ij}$ define an $A_\infty$-structure on $S\s\OO$. Therefore we now just have to work out the signs. Notice that $d_v$ is the total degree of $y_v$ as an element of $\s\OO$ and recall that $d$ is the total degree of $x_1\otimes\cdots\otimes x_j\in A^{\otimes j}$. Therefore, $\varepsilon$ can be written as
\[\varepsilon= i(d-j)+\sum_{1\leq v<w\leq j}k_vd_w.\]
The sign produced by $\mu^{-1}$, as we saw in \Cref{mui}, is precisely determined by the exponent 
\[\sum_{w=2}^jd_w\sum_{v=1}^{w-1}k_v=\sum_{1\leq v<w\leq j}k_vd_w,\]so this sign cancels the right hand summand of $\varepsilon$. This cancellation was expected since this sign comes from $\mu^{-1}$, and operadic composition is defined using $\mu$, see \Cref{insertion}. 
Finally, the sign $(-1)^{i(d-j)}$ left from $\varepsilon$ cancels with $(-1)^{id}$ in \Cref{totsign} and $(-1)^{ij}$ from the isomorphism (\ref{firstiso}). This means that we only need to consider signs produced by vertical shifts. This calculation has been done previously in \Cref{explicit} and as we claimed the result is 
\[M_{ij}(x_1,\dots,x_j)= (-1)^{\sum_{v=1}^j(j-v)(d_v-k_v)}\sum_lSb_j(m_{il};y_1,\dots, y_j). \]

\end{proof}

\begin{remark}\label{equivalent}
Note that as in the case of $A_\infty$-algebras in $\mathrm{C}_R$  
we have two equivalent descriptions of $A_\infty$-algebras in $\tc$.

\begin{itemize}
\item A twisted complex $(A, d^A)$ together with a morphism $\calA_\infty \longrightarrow \uEnd_A$ of operads in $\vbc$, which is determined by a family of elements $M_i \in\utC(A^{\otimes i},A)^{2-i}_0$ for $i \geq 2$ for which the $A_\infty$-relations hold for $i\geq 2$. The composition is the one prescribed by the composition morphisms of $\utC$.
\item A bigraded module $A$ together with a family of elements $M_i \in\ubgMod(A^{\otimes i},A)^{2-i}_0$ for $i \geq 1$ for
which all the $A_\infty$-relations hold. The composition is prescribed by the composition
morphisms of $\ubgMod$.
\end{itemize}
Since the composition morphism
in $\ubgMod$ is induced from the one in $\utC$ by forgetting the differential, these two presentations
are equivalent.
\end{remark}

This equivalence allows us to formulate the following alternative version of \Cref{whitehouse}.
\begin{corollary}\label{alternative}
Given a bigraded module $A$ horizontally bounded on the right we have isomorphisms
\begin{align*}
\Hom_{\mathrm{bgOp}}(d\calA_\infty,\End_A) &\cong
\Hom_{\mathrm{bgOp}}(\calA_\infty, \uEnd_A)\\
&\cong \Hom_{\mathrm{bgOp}}(\calA_\infty, \uEnd_{\Tot(A)})\\
&\cong \Hom_{\mathrm{fOp}}(\calA_\infty,\underline{\End}_{\Tot(A)}),
\end{align*}
where $\mathrm{bgOp}$ is the category of operads of bigraded modules and $\mathrm{fOp}$ is the category of operads of filtered modules. 
\end{corollary}
\begin{proof}
Let us look at the first isomorphism

\[\Hom_{\mathrm{bgOp}}(\calA_\infty, \uEnd_A)\cong \Hom_{\mathrm{bgOp}}(d\calA_\infty,\End_A).\]

Let $f:\calA_\infty\to\uEnd_A$ be a map of operads in $\mathrm{bgOp}$. This is equivalent to maps in $\mathrm{bgOp}$
\[\calA_\infty(j)\to\uEnd_A(j)\]
for each $j\geq 1$, which are determined by elements $M_j\coloneqq f(\mu_j)\in\uEnd_A(j)$ for $v\geq 1$ of bidegree $(0,2-j)$ satisfying the $A_\infty$-equation with respect to the composition in $\ubgMod$. Moreover, $M_j\coloneqq (\tilde{m}_{0j},\tilde{m}_{1j},\dots)$ where $\tilde{m}_{ij}\coloneqq (M_j)_i:A^{\otimes n}\to A$ is a map of bidegree $(i,2-i-j)$. Since the composition in $\ubgMod$ is the same as in $\utC$, the computation of the $A_\infty$-equation becomes analogous to the computation done in \cite[Prop 4.47]{whitehouse}, showing that the maps $m_{ij}=(-1)^i\tilde{m}_{ij}$ for $i\geq 0$ and $j\geq 0$ define a derived $A_\infty$-algebra structure on $A$.

The second isomorphism
\[\Hom_{\mathrm{bgOp}}(\calA_\infty, \uEnd_A)\cong \Hom_{\mathrm{bgOp}}(\calA_\infty, \uEnd_{\Tot(A)})\]
follows from the bigraded module case of \Cref{inverse}. Finally, the isomorphism
\[\Hom_{\mathrm{bgOp}}(\calA_\infty, \uEnd_{\Tot(A)})\cong \Hom_{\mathrm{fOp}}(\calA_\infty,\underline{\End}_{\Tot(A)})\]
is analogous to the last isomorphism of \Cref{whitehouse}, replacing the quasi-free relation by the full $A_\infty$-algebra relations. 
\end{proof}

According to \Cref{alternative}, if we have an $A_\infty$-algebra structure on $A = S\s\OO$, we can consider its arity 1 component $M_1\in\underline{\End}_{\Tot(A)}$ and split it into maps $M_{i1}\in \End_A$. Since these maps must satisfy the derived $A_\infty$-relations, they define a twisted complex structure on $A$. The next corollary describes the maps $M_{i1}$ explicitly.

\begin{corollary}\label{mi1}
Let $\OO$ be a bigraded operad with a derived $A_\infty$-multiplication and let $M_{i1}:S\s\OO\to S\s\OO$ be the arity 1 derived $A_\infty$-algebra maps induced by \Cref{alternative} from $M_1:\Tot(S\s\OO)\to \Tot(S\s\OO)$.
Then \[M_{i1}(x)= \sum_l (Sb_1(m_{il};S^{-1}x)-(-1)^{\langle x,m_{il}\rangle}Sb_1(S^{-1}x;m_{il})),\]
where $x\in (S\s\OO)^{d-k}_k$ and $\langle x,m_{il}\rangle=ik+(1-i)(d-1-k)$.
\end{corollary}
\begin{proof}
The proof of \Cref{alternative} was essentially the same as the proof \Cref{whitehouse}. This means that the proof of this result is an arity 1 restriction of the proof of \Cref{derivedmaps}. Thus, we apply \Cref{totsign} to the case $j=1$. Recall that for $x\in (S\s\OO)^{d-k}_{k}$,
\[M_1(x)=b_1^\star(m;S^{-1}x)-(-1)^{n-1}b_1^\star(S^{-1}x;m).\]
 In this case, there is no $\mu$ involved. Therefore, introducing the final extra sign $(-1)^i$ from the proof of \Cref{derivedmaps}, we get from \Cref{totsign} that
\[\widetilde{M}_{i1}(x)=(-1)^i\sum_l((-1)^{id+i(d-1)} Sb_1(m_{il};S^{-1}x)-(-1)^{d-1+id+k}Sb_1(S^{-1}x;m_{il})),\] where $b_1$ is the brace on $\s\OO$. Simplifying signs we obtain
\[\widetilde{M}_{i1}(x)=\sum_l Sb_1(m_{il};S^{-1}x)-(-1)^{\langle  m_{il},x\rangle}Sb_1(m_{il};S^{-1}x))=M_{i1}(x),\]

where $\langle  m_{il},x\rangle=ik+(1-i)(d-1-k)$, as claimed.
\end{proof}

\section{The Derived Deligne Conjecture}\label{sec:ddeligne}

We can follow a similar process as in \Cref{sec:classicaldeligne} to define higher derived $A_\infty$-maps on the Hochschild complex of a derived $A_\infty$-algebra. More precisely, given an operad $\OO$ with a derived multiplication and $A=S\s\OO$, we will define a derived $A_\infty$-algebra structure on $S\s\End_A$. We will then connect the algebraic structure on $A$ with the structure on $S\s\End_A$ through braces. This connection will allow us to formulate and show a new version of the Deligne conjecture.

Let $\overline{B}_j$ be the bigraded brace map on $\s\End_{S\s\OO}$ and consider the maps
\begin{equation}\label{barbimaps}
\overline{M}'_{ij}:(\s\End_{S\s\OO})^{\otimes j}\to \s\End_{S\s\OO}
\end{equation}
defined as 
\begin{align*}
&\overline{M}'_{ij}(f_1,\dots,f_j)=\overline{B}_j(M_{i\bullet};f_1,\dots, f_j) & j>1,\\
&\overline{M}'_{i1}(f)=\overline{B}_1(M_{i\bullet};f)-(-1)^{ip+(1-i)q}\overline{B}_1(f;M_{i\bullet}),
\end{align*}
for $f$ of natural bidegree $(p,q)$, where $M_{i\bullet}=\sum_j M_{ij}$. We define 
\[\overline{M}_{ij}:(S\s\End_{S\s\OO})^{\otimes j}\to S\s\End_{S\s\OO},\ \overline{M}_{ij} = \overline{\sigma}(M'_{ij})=S\circ M'_{ij}\circ (S^{\otimes n})^{-1}.\]

As in the single-graded case we can define a map $\Phi:S\s\OO\to S\s\End_{S\s\OO}$
as the map making the following diagram commute
\begin{equation}\label{derivedPhi}
\begin{tikzcd}
S\s\OO\arrow[rr, "\Phi"]\arrow[d] & & S\s\End_{S\s\OO}\\
\s\OO\arrow[r, "\Phi'"]& \End_{\s\OO}\arrow[r, "\cong"]& \s\End_{S\s\OO}\arrow[u]
\end{tikzcd}
\end{equation}
where 
\[
\Phi'\colon\s\OO \to \End_{\s\OO},\, x\mapsto \sum_{n\geq 0}b_n(x;-).
\]

In this setting we have the bigraded version of \Cref{theorem}. But before stating the theorem, for the sake of completeness let us state the definition of the Hochschild complex of a bigraded module.
\begin{defin}
We define the \emph{Hochschild cochain complex} of a bigraded module $A$ to be the bigraded module $S\s\End_A$. If $(A,d)$ is a vertical bicomplex, then the Hochschild complex has a vertical differential given by $\partial(f)=[d,f]=d\circ f-(-1)^{q}f\circ d$, where $f$ has natural bidigree $(p,q)$ and $\circ$ is the plethysm corresponding to operadic insertions.
\end{defin}
In particular, $S\s\End_{S\s\OO}$ is the Hochschild cochain complex of $S\s\OO$. If $\OO$ has a derived $A_\infty$-multiplication, then the differential of the Hochschild complex $S\s\End_{S\s\OO}$ is given by $\overline{M}_{01}$ from \Cref{barbimaps}.

The following works in a similar way to \Cref{theorem} but carries the extra index $i$ and using the bigraded sign conventions.
\begin{thm}\label{bigradedtheorem}
The map $\Phi$ defined in diagram (\ref{derivedPhi}) above is a morphism of derived $A_\infty$-algebras, i.e. for all $i\geq 0$ and $j\geq 1$  we have the equation

\[\Phi(M_{ij})=\overline{M}_{ij}(\Phi^{\otimes j}).\]
\qed
\end{thm}

Now that we have \Cref{bigradedtheorem} and the explicit formulas for the derived $A_\infty$-structure on $S\s\OO$, we can deduce the derived version of the Deligne conjecture in an analogous way to how we obtained the $A_\infty$-version in \Cref{ainftydeligne}. In order to do that, we need to first introduce the derived $A_\infty$-version of homotopy $G$-algebras. To have a more succinct formulation we use the notation $\vdeg(x)$ for the vertical degree of $x$.

\begin{defin}\label{derivedJalgebras}
A \emph{derived $J$-algebra} $V$ is a derived $A_\infty$-algebra with structure maps $\{M_{ij}\}_{i\geq 0, j\geq 1}$ such that the shift is $S^{-1}V$ a brace algebra. Furthermore, the braces and the derived $A_\infty$-structure satisfy the following compatibility relations. Let $x, x_1,\dots, x_j, y_1,\dots, y_n\in S^{-1}V$. 
For all $n,i\geq 0$ we demand 

\begin{align*}
(-1)^{\sum_{i=1}^n(n-v)\mathrm{vdeg}(y_v)}Sb_n(&S^{-1}M_{i1}(Sx);y_1,\dots, y_n)=\\
\underset{\mathclap{1\leq i_1\leq n-k+1}}{\sum_{\mathclap{l+k-1=n}}}(-1)^{\varepsilon}M_{il}(Sy_1,\dots, Sb_{k}(x;y_{i_1},\dots),\dots, Sy_n)\\
-(-1)^{\langle x,M_{il}\rangle}\underset{\mathclap{1\leq i_1\leq n-k+1}}{\sum_{\mathclap{l+k-1=n}}}&(-1)^{\eta} Sb_k(x;y_1,\dots, S^{-1}M_{il}(Sy_{i_1},\dots,), \dots, y_n)
\end{align*}
where

\begin{align*}
\varepsilon = \sum_{v=1}^{i_1-1}\langle Sy_v,S^{1-k}x\rangle&+\sum_{v=1}^{k}\vdeg(y_{i_1+v-1})(k-v)+(l-i_1)\vdeg(x).
\end{align*}
and
\begin{align*}
\eta=& \sum_{v=1}^{i_1-1}(k-v)\vdeg(y_v)+l\sum_{v=1}^{i_1-1}\vdeg(y_v)
+\sum_{v=i_1}^{i_1+l-1}(k-i_1)\vdeg(y_v)+\sum_{v=i_1}^{n-l}(k-v)\vdeg(y_{v+l}).
\end{align*}

For $j>1$ we demand
\begin{align*}
&(-1)^{\sum_{i=1}^n(n-v)\mathrm{vdeg}(y_v)}Sb_n(S^{-1}M_{ij}(Sx_1,\dots, Sx_j);y_1,\dots, y_n)=\\
&\sum(-1)^{\varepsilon}M_{il}(Sy_1,\dots, Sb_{k_1}(x_1;y_{i_1},\dots),\dots, Sb_{k_j}(x_j;y_{i_j},\dots),\dots, Sy_n).
\end{align*}
The unindexed sum runs over all possible choices of non-negative integers that satisfy $l+k_1+\cdots+k_j-j=n$ and over all possible ordering preserving insertions. The right hand side sign is given by 
\begin{align*}
\varepsilon =&\underset{1\leq v\leq k_t}{\sum_{\mathclap{1\leq t\leq j}}} \vdeg(y_{i_t+v-1})(k_v-v)+ \sum_{\mathclap{1\leq i< l\leq j}}k_v\vdeg(x_l)+\underset{\mathclap{i_{t}\leq v< i_{t+1}}}{\sum_{\mathclap{0\leq t< l\leq j}}}\langle Sy_v,S^{1-k_l}x_l\rangle\\
&+\sum_{\mathclap{0\leq v<l\leq j}}(i_{v+1}-i_v-k_v)\vdeg(S^{1-k_l}x_l)+\sum_{\mathclap{1\leq v\leq l\leq j}} \vdeg(x_v)(i_{l+1}-i_l-k_l)
\end{align*}
All the above shifts are vertical and we are setting $i_0=0$, $i_{j+1}=n+1$. 
\end{defin}

For our final result we can now apply 
\Cref{bigradedtheorem} analogously to \Cref{ainftydeligne} using the explicit expressions and signs given by \Cref{derivedmaps}, \Cref{mi1} and \Cref{bigradedsign}. This gives us the Derived Deligne Conjecture, which explicitly describes the structure carried by the Hochschild complex of a derived $A_\infty$-algebra.

\begin{corollary}[The derived Deligne conjecture]\label{dainftydeligne}
If $A$ is a derived $A_\infty$-algebra horizontally bounded on the right, then its Hochschild complex $S\s\End_A$ is a derived $J$-algebra. \qed
\end{corollary}

\appendix

\section{Twisted complex on an operad}\label{twistedoperad}
In this section we provide a description of the twisted complex structure on an operad $\OO$ with a derived $A_\infty$-multiplication. More precisely, we show by hand that the maps found in \Cref{mi1} define a twisted complex structure on $S\s\OO$.

\begin{lem}\label{twistedmaps}
Let $\OO$ be an operad with a derived $A_\infty$-multiplication $m\in\s\OO$. Then $S\s\OO$ becomes a twisted complex with structure maps
\[M_{i1}(x)= \sum_l (Sb_1(m_{il};S^{-1}x)-(-1)^{\langle x,m_{il}\rangle}Sb_1(S^{-1}x;m_{il})),\]
where $x\in (S\s\OO)^{n-k}_k$ and $\langle x,m_{il}\rangle=ik+(1-i)(n-1-k)$.
\end{lem}
\begin{proof}

Througout the proof we omit the shift maps. Let us first check the twisted complex equation up to signs, to give a conceptual proof before introducing the signs. Up to sign, the maps  $\{M_{i1}\}_{i\geq 0}$ must satisfy the equation
\[\sum_{i+j=u} M_{i1}\circ M_{j1}=0,\]
for all $u$, where $\circ$ is composition of maps. 
Therefore, up to signs we have to compute
\begin{align*}
\sum_{i+j=u}M_{i1}(M_{j1}(x))=&\sum_{i+j=u}M_{i1}\left(\sum_l b_1(m_{jl};x)+b_1(x;m_{jl})\right)\\
=&\sum_{i+j=u}\sum_{l,k}\left(b_1(m_{ik}; b_1(m_{jl};x))+b_1(m_{ik};b_1(x;m_{jl}))\right.\\
&\left.+b_1(b_1(m_{jl};x);m_{ik})+b_1(b_1(x;m_{jl});m_{ik})\right).
\end{align*}
Applying the brace relation we obtain
\begin{align*}
\sum_{i+j=u}\sum_{l,k}(b_1(m_{ik}; b_1(m_{jl};x))+b_1(m_{ik};b_1(x;m_{jl}))+\\
 b_2(m_{jl};x,m_{ik})+b_1(m_{jl};b_1(x;m_{ik}))+b_2(m_{jl};m_{ik},x)+\\
b_2(x;m_{jl},m_{ik})+b_1(x;b_1(m_{jl};m_{ik}))+b_2(x;m_{ik},m_{jl})).
\end{align*}

In the sum, all terms of the form $b_1(x;b_1(m_{jl};m_{ik}))$ that can be seen in the last line should add up to vanish provided that $m$ is a $dA_\infty$-multiplication, meaning that up to sign $b_1(m;m)=0$. 
 Since $i$ and $j$ are interchangeable, i.e. for each pair $(i,j)$ there is the pair $(j,i)$, the terms $b_2(x;m_{jl},m_{ik})+b_2(x;m_{ik},m_{jl})$ in the last line should cancel as well. For this, we should have the pair $(j,i)$ with the opposite sign. Here it is also relevant that the sum runs through all possible values of $k$ and $l$, so that the pair $(j,i)$ appears with $l$ and $k$ interchanged as well. So far the entire last line vanishes up to sign.

Then $b_1(m_{ik};b_1(x;m_{jl}))$ on the first line should cancel with $b_1(m_{jl};b_1(x;m_{ik}))$ on the second line, but from a different summand: the one where $i$ and $j$ are interchanged. Finally, the remaining terms $b_1(m_{ik}; b_1(m_{jl};x))+b_2(m_{jl};x,m_{ik})+b_2(m_{jl};m_{ik},x)$ add up to $b_1(b_1(m;m);x)$ up to sign. That would cancel everything.

Let us now introduce the signs. We now compute for all $u$
\[\sum_{i+j=u} (-1)^iM_{i1}\circ M_{j1}.\]
For $x\in\s\OO$, by definition, we have
\begin{align*}
\sum_{i+j=u}(-1)^iM_{i1}(M_{j1}(x))=\sum_{i+j=u}(-1)^iM_{i1}\left(\sum_l b_1(m_{jl};x)-(-1)^{\langle x,m_{jl}\rangle}b_1(x;m_{jl})\right)=\\
\sum_{i+j=u}(-1)^i\sum_{l,k}\left(b_1(m_{ik}; b_1(m_{jl};x))-(-1)^{\langle x,m_{jl}\rangle}b_1(m_{ik};b_1(x;m_{jl}))+\right.\\
\left. -(-1)^{\langle b_1(m_{jl};x),m_{ik}\rangle}b_1(b_1(m_{jl};x);m_{ik})+(-1)^{\langle b_1(m_{jl};x),m_{ik}\rangle+\langle x|m_{jl}\rangle}b_1(b_1(x;m_{jl});m_{ik})\right).
\end{align*}
Observe that $\langle b_1(m_{jl};x),m_{ik}\rangle=\langle m_{ij},m_{ik}\rangle+\langle x,m_{ik}\rangle$.

Applying the brace relation we obtain

\begin{align}\label{twistedequation}
\sum_{i+j=u}\sum_{l,k}((-1)^ib_1(m_{ik}; b_1(m_{jl};x))-(-1)^{i+\langle x,m_{jl}\rangle}b_1(m_{ik};b_1(x;m_{jl}))+\nonumber\\
 -(-1)^{i+\langle b_1(m_{jl};x),m_{ik}\rangle}(b_2(m_{jl};x,m_{ik})+(-1)^{\langle x,m_{ik}\rangle}b_2(m_{jl};m_{ik},x))\nonumber\\
 -(-1)^{i+\langle b_1(m_{jl};x),m_{ik}\rangle}b_1(m_{jl};b_1(x;m_{ik}))\nonumber\\
+(-1)^{i+\langle b_1(m_{jl};x),m_{ik}\rangle+\langle x,m_{jl}\rangle}(b_2(x;m_{jl},m_{ik})+(-1)^{\langle m_{ik},m_{jl}\rangle}b_2(x;m_{ik},m_{jl}))\nonumber\\
+(-1)^{i+\langle b_1(m_{jl};x),m_{ik}\rangle+\langle x,m_{jl}\rangle}b_1(x;b_1(m_{jl};m_{ik}))).
\end{align}

Recall from \Cref{sharp} that $m$ being a $dA_\infty$-multiplication means that \[\sum_{i+j=u}\sum_{k,l}(-1)^ib_1(m_{jl};m_{ik})=0.\] 
Let us check now the cancellations with the signs. First, let us check that the terms 

\[(-1)^{i+\langle b_1(m_{jl};x),m_{ik}\rangle+\langle x,m_{jl}\rangle}b_1(x;b_1(m_{jl};m_{ik})))\]

can be added up to vanish. For that, we compute the sign 

\[
\langle b_1(m_{jl};x),m_{ik}\rangle+\langle x,m_{jl}\rangle=\langle m_{jl},m_{ik}\rangle+\langle x,m_{ik}\rangle+\langle x,m_{jl}\rangle.\]

Recall that the braces are defined on the operadic suspension, so that the bidegree of $m_{ik}$ is $(i,1-i)$. Therefore, writing the bidegree of $x$ as $(k,n-k)$, so that the total degree is $|x|=n$, the above equals 

\begin{align*}
&ji+(1-i)(1-j)+ki+(n-k)(1-i)+kj+(n-k)(1-j)\\
&\equiv 1+i+j + (i+j)k+(i+j)(n-k)\mod 2\\
&=1+(i+j)(1+n)=1+u(1+|x|).
\end{align*}

Since this sign is constant for all terms $b_1(m_{ik};m_{ij})$ that share the same horizontal degree $i+j=u$, we can rewrite

\[(-1)^{i+\langle b_1(m_{jl};x),m_{ik}\rangle+\langle x,m_{jl}\rangle}b_1(x;b_1(m_{jl};m_{ik})))=-(-1)^{u(1+|x|)}b_1(x;(-1)^ib_1(m_{ik};m_{jl})).\]
Hence, 

\[\sum_{i+j=u}\sum_{k,l}-(-1)^{u(1+|x|)}b_1(x;(-1)^ib_1(m_{ik};m_{jl}))=0.\]
Therefore, after applying the brace relation expression (\ref{twistedequation}) reduces to

\begin{align}\label{twistedequation2}
\sum_{i+j=u}\sum_{l,k}((-1)^ib_1(m_{ik}; b_1(m_{jl};x))-(-1)^{i+\langle x,m_{jl}\rangle}b_1(m_{ik};b_1(x;m_{jl}))+\nonumber\\
 -(-1)^{i+\langle b_1(m_{jl};x),m_{ik}\rangle}(b_2(m_{jl};x,m_{ik})+(-1)^{\langle x,m_{ik}\rangle}b_2(m_{jl};m_{ik},x))\nonumber\\
 -(-1)^{i+\langle b_1(m_{jl};x),m_{ik}\rangle}b_1(m_{jl};b_1(x;m_{ik}))\nonumber\\
+(-1)^{i+\langle b_1(m_{jl};x),m_{ik}\rangle+\langle x,m_{jl}\rangle}(b_2(x;m_{jl},m_{ik})+(-1)^{\langle m_{ik},m_{jl}\rangle}b_2(x;m_{ik},m_{jl})).
\end{align}

Let us focus on the last line. For each pair $(i,j)$ we should have cancellation with the pair $(j,i)$, which adds the same elements, but with different signs. We also need to consider the pairs $(k,l)$ and $(l,k)$ to get a cancellation. Let us compare the signs. For the pair $((i,j),(k,l))$ we have precisely the last line of the above equation

\[(-1)^{i+\langle b_1(m_{jl};x),m_{ik}\rangle+\langle x,m_{jl}\rangle}(b_2(x;m_{jl},m_{ik})+(-1)^{\langle m_{ik},m_{jl}\rangle}b_2(x;m_{ik},m_{jl}))\]

For the pair $((j,i),(l,k))$ we have
\[(-1)^{j+\langle b_1(m_{ik};x),m_{jl}\rangle+\langle x,m_{ik}\rangle}(b_2(x;m_{ik},m_{jl})+(-1)^{\langle m_{jl},m_{ik}\rangle}b_2(x;m_{jl},m_{ik})).\]
 Comparing the sign of $b_2(x;m_{jl},m_{ik})$ we find that for $((i,j),(k,l))$ we have

\[-(-1)^{i+(i+j)(1+|x|)}b_2(x;m_{jl},m_{ik})=-(-1)^{j+u|x|}b_2(x;m_{jl},m_{ik})\]

and for the pair $((j,i),(l,k))$ we have

\[(-1)^{j+u|x|}b_2(x;m_{jl},m_{ik}).\]

As we see, we get opposite signs and thus cancellation. For $b_2(x;m_{ik},m_{jl})$ it is completely analogous. Thus, we have reduced expression (\ref{twistedequation2}) to

\begin{align}\label{twistedequation3}
\sum_{i+j=u}\sum_{l,k}((-1)^ib_1(m_{ik}; b_1(m_{jl};x))-(-1)^{i+\langle x,m_{jl}\rangle}b_1(m_{ik};b_1(x;m_{jl}))+\nonumber\\
 -(-1)^{i+\langle b_1(m_{jl};x),m_{ik}\rangle}(b_2(m_{jl};x,m_{ik})+(-1)^{\langle x,m_{ik}\rangle}b_2(m_{jl};m_{ik},x))\nonumber\\
 -(-1)^{i+\langle b_1(m_{jl};x),m_{ik}\rangle}b_1(m_{jl};b_1(x;m_{ik})).
\end{align}

In a similar fashion to the previous calculation, we are going to cancel $b_1(m_{ik};b_1(x;m_{jl}))$ in the first line with $b_1(m_{jl};b_1(x;m_{ik}))$ in the last line by considering switched pairs. For the pair $((i,j),(k,l))$, the term in the first line is 

\[-(-1)^{i+\langle x,m_{jl}\rangle}b_1(m_{ik};b_1(x;m_{jl}))\]

and for the pair $((j,i),(l,k))$ the term in the last line is

\begin{align*}
-(-1)^{j+\langle b_1(m_{ik};x),m_{jl}\rangle}b_1(m_{ik};b_1(x;m_{jl}))&=(-1)^{1+j+\langle m_{ik},m_{jl}\rangle+\langle x,m_{jl}\rangle}b_1(m_{ik};b_1(x;m_{jl}))\\
&=(-1)^{i+\langle x,m_{jl}\rangle}b_1(m_{ik};b_1(x;m_{jl})),
\end{align*}

which has precisely the opposite sign to the other pair, and thus cancels. This reduces expression (\ref{twistedequation3}) to 

\begin{align}\label{twistedequation4}
\sum_{i+j=u}\sum_{l,k}((-1)^ib_1(m_{ik}; b_1(m_{jl};x))&\nonumber\\
 -(-1)^{i+\langle b_1(m_{jl};x),m_{ik}\rangle}(b_2(m_{jl};x,m_{ik})&+(-1)^{i+\langle m_{jl},m_{ik}\rangle}b_2(m_{jl};m_{ik},x)).
\end{align}

We want these terms to add up to something of the form $b_1(b_1(m;m);x)$. Notice that for this we need to switch some pairs. For simplicity, we switch the pair of the first term and rewrite the sum as

\begin{align*}
\sum_{i+j=u}\sum_{l,k}((-1)^jb_1(m_{jl}; b_1(m_{ik};x))&\\
 -(-1)^{i+\langle b_1(m_{jl};x),m_{ik}\rangle}b_2(m_{jl};x,m_{ik})&+(-1)^{i+\langle m_{jl}, m_{ik}\rangle}b_2(m_{jl};m_{ik},x)).
\end{align*}

Simplifying the signs we get

\begin{align*}
\sum_{i+j=u}\sum_{l,k}((-1)^jb_1(m_{jl}; b_1(m_{ik};x))
 +(-1)^{j+\langle x,m_{ik}\rangle}b_2(m_{jl};x,m_{ik})+(-1)^{j}b_2(m_{jl};m_{ik},x)).
\end{align*}

By the brace relation and \Cref{sharp} this equals

\[\sum_{i+j=u}\sum_{l,k}(-1)^jb_1(b_1(m_{jl};m_{ik});x)=0.\]
\end{proof}



\newcommand{\etalchar}[1]{$^{#1}$}

\end{document}